\documentclass[11pt,           
final,                         
]{article}

\usepackage{amsmath}           
\usepackage{amssymb}           
\usepackage[a4paper]{geometry} 
\usepackage{setspace}          
\usepackage{bbm}               
\usepackage{mathtools}         
\usepackage{stmaryrd}          
\usepackage[amsmath,thmmarks,hyperref]{ntheorem} 
\usepackage{wasysym}           

\author{
  \textbf{Konrad Schm\"udgen and Matthias Sch\"otz}\\
  Mathematisches Institut\\
  Universit\"at Leipzig
}

\makeatletter

\newcommand{\refitem}[1] {\textit{(\ref{#1})}}

\numberwithin{equation}{section}

\theoremheaderfont{\normalfont\bfseries}
\theorembodyfont{\itshape}
\newtheorem{lemma}{Lemma}[section]

\newtheorem{proposition}[lemma]{Proposition}
\newtheorem{theorem}[lemma]{Theorem}
\newtheorem{corollary}[lemma]{Corollary}
\newtheorem{definition}[lemma]{Definition}
\theorembodyfont{\rmfamily}

\theoremsymbol{\hbox{$\fullmoon$}}
\newtheorem{remark}[lemma]{Remark}
\newtheorem{example}[lemma]{Example}

\theoremheaderfont{\normalfont\bfseries}
\theorembodyfont{\normalfont}
\theoremstyle{nonumberplain}
\theoremseparator{:}
\theoremsymbol{\hbox{$\boxempty$}}
\newtheorem{proof}{Proof}

\newcommand{\I}              {\mathrm{i}}
\newcommand{\Unit}           {\mathbbm{1}}
\newcommand{\argument}       {\ignorespaces{\,\cdot\,}\ignorespaces}
\newcommand{\abs}[2][]        {#1\lvert#2#1\rvert}
\newcommand{\norm}[2][]        {#1\lVert#2#1\rVert}
\newcommand{\CC}{\mathbbm{C}}
\newcommand{\RR}{\mathbbm{R}}
\newcommand{\NN}{\mathbbm{N}}

\newcommand{\set}[3][]{#1\{\,#2 \;#1|\; #3\,#1\}}

\newcommand{\bd}{{\mathrm{bd}}}
\newcommand{\characters}{\mathrm{char}}
\newcommand{\mloc}{\mathrm{\textup{-}loc}}
\newcommand{\Lin}{\mathrm{Lin}}
\newcommand{\Pos}{\mathrm{Pos}}
\renewcommand{\sb}{\mathrm{sb}}
\newcommand{\msb}{\mathrm{\textup{-}sb}}
\newcommand{\D}{\mathrm{d}}
\newcommand{\J}{\mathcal{J}}
\newcommand{\K}{\mathcal{K}}
\newcommand{\A}{\mathcal{A}}
\renewcommand{\S}{\mathcal{S}}
\newcommand{\B}{\mathcal{B}}
\newcommand{\C}{\mathcal{C}}
\newcommand{\T}{\mathcal{T}}
\newcommand{\Q}{\mathcal{Q}}
\newcommand{\cH}{\mathcal{H}}

\title{Positivstellens\"atze for Semirings}

\date{\today}

\begin{document}
\begin{onehalfspace}

\maketitle

\begin{abstract}
  In this paper we develop a number of results and notions concerning Positivstellens\"atze for semirings (preprimes)
  of commutative unital real algebras. First we reduce the Archimedean Positivstellensatz for semirings to the corresponding
  result for quadratic modules. Various applications of the Archimedean Positivstellensatz for semirings are investigated. A general Positivstellensatz with denominators is proved for filtered algebras with semirings.
  As an application we derive a denominator-free Positivstellensatz for the cylindrical extension of an algebra with Archimedean semiring.
  A large number of illustrating examples are given.\\[-0.6em]
 

\end{abstract}


\section{Introduction}

The Archimedean Positivstellensatz is a fundamental result of real algebraic geometry. There exist a version for Archimedean quadratic
modules and another one for Archimedean semirings (preprimes), see e.g.  \cite[Theorem 5.4.4]{marshall} or \cite[Theorem 12.35]{sch17}.
The case of  quadratic modules and of preorderings is extensively studied in the literature, see e.g. \cite{jacobi, kumsch, cimpric09, cmn, marshall}.
In contrast, the case of semirings is less understood and 
no Hilbert space proof
seems to be known up to now. In general, squares of elements do not belong to a semiring. Hence a linear functional which is nonnegative on a semiring
is not necessarily a positive functional, 
so the GNS construction does not apply. 

The aim of this paper is to study Positivstellens\"atze for semirings.
Among others, we give a new approach to, and new applications of the Archimedean Positivstellensatz for semirings and we prove two new  Positivstellens\"atze for Non-Archimedean semirings.

The starting point is the observation  that for each Archimedean semiring $\S$ of $\A$,  
\begin{equation*}
    \S^\dagger \coloneqq \set[\big]{ a\in \A }{ a+\epsilon \Unit \in \S ~ \textup{ for all }~  \epsilon \in {(0,\infty)} }
\end{equation*}
is an Archimedean preordering (Theorem~\ref{archpre}). As a consequence, 
the Archimedean Positivstellensatz
for semirings follows at once from its counter-part for quadratic modules (Theorem~\ref{archpos}). Further, we develop a number of applications the Archimedean Positivstellensatz
for semirings to various situations.

The second half of the paper is devoted to Positivstellens\"atze for general semirings. Our approach is motivated by  M. Marshall's  extension \cite{M2}
of the Archimedean Positivstellensatz for  preorderings to the non-compact case.  Fix a polynomial $p$  which dominates the coordinate functions on a
semialgebraic set $\K$. Then \cite[Corllary 3.1]{M2} states  that for  nonnegative $f$ on $\K$ and sufficiently large $m\in \NN$ and any $\varepsilon >0$ there exists $\ell\in \NN$
such that $p^\ell(f+\varepsilon p^m)$ is in the prerordering; the case $p=1$ gives the Positivstellensatz in the compact case.

A  main result in this context is a general Positivstellensatz for a filtered algebra $\A$ and a semiring $\S$ with a cofinal element $s$ which is
compatible with the filtration $(\A^{(k)})_{k\in \NN_0}$.
The new ingredient is a symbol algebra 
$\A^\sb$ with semiring $\S^\sb$ that allows one to control the behaviour of elements at ''infinity". This theorem states that an element 
$a\in \A^{(k)}$ is strictly positive on $\S$-positive characters of $\A$ and on $\S^\sb$-positive characters of  $\A^\sb$ if and only if there exist 
$\ell\in \NN$, $\varepsilon >0$ such that $s^\ell a\in \varepsilon s^{k+\ell}+ \S$. 

The element $s^\ell$ plays the role of a denominator in the filtered  Positivstellensatz. We introduce the notion of an $s$-localizable semiring in order
to remove the denominator. Using this concept we derive another main result of the paper, a denominator-free Positivstellensatz for cylindrical extensions
of an algebra  with an Archimedean semiring (Theorem \ref{theorem:streifen}). A similar result for preorderings was obtained by V.~Powers \cite{powers}.
Also, we note that there is an important result of Marshall \cite{marshall10} about the description of  nonnegative polynomials on the  strip $[0,1]\times \RR$.

This paper is organized as follows. In Section \ref{basicdef} we collect basic definitions and facts used throughout this paper. In Section \ref{archsemirings}
we present our approach to the Archimedean Positivstellensatz for semirings. In Sections \ref{sharpening} and  \ref{sharpenpolyhedra} we use the 
Archimedean Positivstellensatz for semirings to derive  sharpenings of  well-known  Archimedean Positivstellens\"atze 
(Theorem \ref{auxsemiring} and Corollary~\ref{corollary:polyhedron1}) that might be useful in polynomial optimization.
In Section \ref{convexsets}  we apply the Archimedean Positivstellensatz to express positive polynomials on  compact {\it convex}
sets in terms of supporting polynomials. Section \ref{examplesarch} contains some examples and other applications of the Archimedean
Positivstellensatz for semirings. 

Section \ref{cofinal}  begins the study of the Non-Archimedean case. We develop two  useful concepts, cofinal elements
and localizable semirings. Cofinal elements will appear as denominators in the filtered Positivstellensatz and localizability will be used to remove denominators. 
In Section \ref{filtered} we investigate filtered algebras with semirings and define the corresponding symbol algebras and 
semirings. These notions enter in an essential way into the filtered Positivstellensatz proved in this section.  In Section \ref{examples}
we develop some examples for the filtered Positivstellensatz. In the last Section \ref{onedimextension} we consider a finitely
generated algebra $\A$  with an Archimedean semiring $\S$ and derive a denominator-free Positivstellensatz for the algebra 
$\A \otimes \RR[z]$ and the semiring $\S \otimes \Pos(J)$, where $\Pos(J)$ denotes the semiring of polynomials in $z$ that are
pointwise non-negative on an unbounded closed subset $J$ of $\RR$.

There are only very few   applications of Positivstellens\"atze for  semirings in the literature, unlike the case of quadratic
modules. For this reason we have included a large number of examples, explicit formulas and applications  to illustrate the results.

In general, denominator-free Positivstellens\"atze provide solvability criteria for the moment problem. Some of such applications are 
explicitely stated in Theorem \ref{auxsemiring} and Corollaries~\ref{corollary:polyhedron1} and \ref{onedimmp}.

\smallskip

Throughout this paper, $\A$ denotes a {\bf commutative unital $\RR$-algebra} and $\S$ is a {\bf  semiring} of $\A$.

For  real algebraic geometry we  refer to the  books \cite{PD}, \cite{marshall}, and the survey \cite{scheiderer}.

\section{Basic Definitions and Facts}\label{basicdef}
By a {\it convex cone} in $\A$ we mean a non-empty subset $\C$ of $\A$ such that $\C+\C\subseteq \C$ and $\lambda\cdot \C\subseteq \C$
for all $\lambda \geq 0$. If $\C$ is a convex cone in $\A$, then $\A$ becomes an ordered real vector space with respect to
the order relation $\preceq_\C$ defined by
\begin{equation} \label{preceq}
  a \preceq_\C b ~~ \textup{ if and only if} ~~ b-a\in \C.
\end{equation}
If no confusion can arise, we omit the subscript $\C$ and write simply $\preceq$.

\begin{definition}
  A convex cone $\C$ in $\A$ is called a\\
  $\bullet$\, \emph{quadratic module} of\, $\A$ if\, $\Unit\in \C$ and\, $a^2b \in \C$  for all $b\in \C$ and $a\in \A.$\\
  $\bullet$\, \emph{semiring} of\, $\A$ if\, $\Unit\in \C$ and\, $\C\cdot \C\subseteq \C$.\\
  $\bullet$\, \emph{preordering} of $\A$ if\, $\C$ is a quadratic module and a semiring.
\end{definition}

\begin{definition}
  Let $\S$ be a semiring of $\A$. A convex cone $\C$ in $\A$ is called an \emph{$\S$-module}
  if\, $\Unit\in \C$ and\, $ac\in \C$ for all $a\in \S$ and $c\in \C$.
\end{definition}
Note that $\S \subseteq \C$ for every $\S$-module $\C$.

In the literature ``semirings" are often called ``preprimes". 
Note that semirings are closed  under multiplication of elements and quadratic modules are invariant under multiplication by squares. 

The smallest preordering of $\A$ is the cone\, $\sum \A^2\coloneqq \set[\big]{ \sum_{i=1}^k a_i^2 }{ a_i\in {\A},\,  k\in \NN }$
of all finite sums of squares of elements of $\A$.

Let $f_1,\dots,f_r\in \A$. Then
\begin{align}\label{semiringf}
  \S
  =
  \S(f_1,\dots,f_r)
  \coloneqq
  \set[\Big]{ 
    \sum\nolimits_{n_1,\dots,n_r=0}^k \alpha_{n_1,\dots,n_r} f_1^{n_1}\cdots f_r^{n_r}
  }{ 
    \alpha_{n_1,\dots,n_r}\in [0,+\infty), k\in \NN_0
  }
\end{align}
is the {\it semiring generated by}\, $f_1,\dots,f_r$, 
and
\begin{align}\label{preorderingf}
  \T
  =
  \T(f_1,\dots,f_r)
  \coloneqq
  \set[\Big]{
    \sum\nolimits_{e=(e_1,\dots,e_k)\in \{0,1\}^k} f_1^{e_1}\cdots f_k^{e_k} \sigma_e
  }{
    \sigma_e\in \sum\A^2
  }
\end{align}
is the {\it preordering generated by}\, $f_1,\dots,f_r$.

\begin{definition}
  Let $\C$ be a convex cone in $\A$ such that $\Unit\in \A$. We define
  \begin{equation}\label{archcondition}
    \A^{\bd}(\C)
    \coloneqq
    \set[\big]{
      a\in \A
    }{
      \textup{there exists }~  \lambda \in (0,\infty)~ \textup{ such that } ~ (\lambda \Unit- a)\in \C ~ \textup{and}~ (\lambda\Unit +a)\in \C
    }
  \end{equation}
  The cone $\C$ is called \emph{Archimedean} if\, $\A^{\bd}(\C)=\A$.
\end{definition}

We note the following fact.

\begin{lemma} \label{archhilfs}
  Suppose $\S$ is a semiring of $\A$. If $a,b\in \A$ and $c,d\in \S$ fulfill
  $-c \preceq a \preceq c$ and $-d \preceq b \preceq d$, then $-cd \preceq ab \preceq cd$.
  In particular, $\A^\bd(\S)$ is a unital subalgebra of $\A$.
  
  Let\, $a_j, j\in J,$ be elements of $\A$. Suppose that the set $\{\Unit,a_j: j\in J\}$ generates the algebra $\A$. 
  Then $\S$ is Archimedean if and only if there exists  numbers $\alpha_j>0, \beta_j>0$ such that  $(\alpha_j\Unit + a_j)\in \S$ and $(\beta_j\Unit -a_j)\in \S$ for $j\in J$.
\end{lemma}
\begin{proof}
  The first assertion follows immediately from the identities
  \begin{align*}
    cd + ab
    =
    \frac{(c+a)(d+b) + (c-a)(d-b)}{2}
    \in
    \S
    ~~\textup{and}~~
    cd - ab
    =
    \frac{(c+a)(d-b) + (c-a)(d+b)}{2}
    \in
    \S
    .
    ~~
  \end{align*}
  The other assertions are now easily checked.
\end{proof}

In real algebraic geometry, $\A^\bd(\S)$\, is called the {\it algebra of\, $\S$-bounded elements} of $\A$, see e.g. \cite{marshall}, \cite{scheiderer}, \cite{schw02}.

Lemma \ref{archhilfs} is  useful if one wants to show that a semiring is Archimedean:
It suffices to verify the Archimedean condition in \eqref{archcondition} for a set of algebra generators.

We denote by $\characters(\A)$ the  characters  of $\A$, that is, $\characters(\A)$ is the set of unital algebra homomorphisms
$\varphi\colon\A\mapsto \RR$. For subsets $\C$ of $\A$ and $\K$ of $\characters(\A)$ we define
\begin{align}
  \characters (\A,\C)
  &\coloneqq
  \set[\big]{ \varphi\in \characters(\A) }{ \varphi(a)\geq 0 ~\textup{ for all }~ a\in \A },\\ 
  \label{eq:posJ}
  \Pos (\K)
  &\coloneqq
  \set[\big]{ a\in \A }{ \varphi(a)\geq 0 ~ \textup{ for all }~ \varphi\in \K }.
\end{align} 
 
We equip the algebraic dual $\A^*$ of the vector space $\A$ with the weak topology with respect to the functions
$\varphi \mapsto \varphi(a), \varphi\in \A^*, a\in \A$. These functions are continuous on $\A^*$ by definition.
It is well known that $\characters (\A,\C)$ is compact in the weak topology of $\A^*$ if $\C$ is an Archimedean convex cone. 

If $\A$ consists of functions on a set $X$, then for any $t\in X$ the {\it evaluation functional} $\chi_t$ defined by $\chi_t(f)\coloneqq f(t), f\in \A$, is a character of $\A$.

\begin{example}
  Let $\A$ be the polynomial algebra $\RR[x_1,\dots,x_d]$ and   $f_1,\dots,f_r\in \A$, where $r\in\NN$.  
  Let $\S$ denote the semiring \eqref{semiringf} and 
  $\T$ the preordering \eqref{preorderingf} of $\A$  generated by $f_1,\dots,f_r$. Then\,
  $\characters (\A,\S) = \characters (\A,\T)$ is the set of evaluation functionals
  $\chi_t$ at the points of the semialgebraic set 
  \begin{align}\label{semiset}
    \K(f_1,\dots,f_r) \coloneqq \set[\big]{ t\in\RR^d }{ f_1(t)\geq 0,\dots,f_r(t)\geq 0 }.
  \end{align}
  Sometimes we  identify points $t\in\RR^d$ with the corresponding point evaluations $\chi_t$.
\end{example}

\begin{definition}
  For $j=1,\dots,n$, let $\A_j$ be a  commutative unital real algebra and $\S_j$ a semiring of $\A_j$.
  Then we define the tensor product $\S$ of these semirings as the subset
  \begin{align}\label{tensorpr}
    \S
    \coloneqq
    \set[\Big]{
      \sum\nolimits_{i=1}^k c_{1i}\otimes c_{2i}\otimes\cdots\otimes c_{ni}
    }{
      c_{ji}\in \S_j ~ \textup{for}~ j=1,\dots,n; i=1,\dots,k~ \textup{ and }~ k\in \NN
    }
  \end{align}
  of the commutative unital real algebra $\A=\A_1\otimes\cdots\otimes\A_n$.
\end{definition}
It is easily verified that $\S$ as in \eqref{tensorpr} is indeed a semiring of $\A=\A_1\otimes\cdots\otimes\A_n$.
In general, $\S$ is not invariant under multiplication by squares of the algebra $\A$,
so $\S$ is only a semiring  even if each $S_j$ is a preordering.

We mention two  special cases of the construction  \eqref{tensorpr} which will appear again in later sections. 
The following  are ``natural" semirings of the polynomial algebra $\A\coloneqq\RR[x,y]\equiv \RR[x]\otimes \RR[y]$:
\begin{align*}
  \S_1
  &=
  \set[\Big]{
    \sum\nolimits_{j=1}^k
    \big((1-x^2)p_{j1}(x) +p_{j2}(x)\big)\, q_j(y)
  }{
    p_{j1},p_{j2}\in \sum\RR[x]^2,\, q_j\in \sum \RR[y]^2,\,k\in\NN
  },\\
  \S_2
  &=
  \set[\Big]{
    \sum\nolimits_{k,\ell=0}^m (1-x)^k (1+x)^\ell q_{k\ell}(y) 
  }{
    q_{kl}\in \sum\RR[y]^2,\, m\in \NN_0
  },
\end{align*}
and $\characters (\A,\S_1)=\characters (\A,\S_2)$ is the strip $[-1,1]\times \RR$ in $\RR^2$. Neither $\S_1$ nor $\S_2$ is a quadratic module.

The following notion is crucial throughout this paper.
\begin{definition}
  For a convex cone $\C$ in $\A$  such that $\Unit \in \C$, we set
  \begin{equation}
    C^\dagger \coloneqq \set[\big]{ a\in \A }{ \textup{there exists }~ c\in \C~ \textup{ such that }~ a+\epsilon c \in \C~ \textup{ for all }~ \epsilon \in {(0,\infty)} }.
    \label{eq:ddagger}
  \end{equation}
\end{definition}
\begin{lemma}
Suppose $\C$ is a convex cone  with $\Unit\in \C$. If $\C$ is Archimedean, then \eqref{eq:ddagger} simplifies to
\begin{equation}\label{dagger1}
  C^\dagger = \set[\big]{ a\in \A }{ a+\epsilon \Unit \in \C~ \textup{ for all }~ \epsilon \in {(0,\infty)} }.
\end{equation}
\end{lemma}
\begin{proof} Suppose that $a$ and $c$ are as in \eqref{eq:ddagger}. Let\, $\epsilon \in (0,+\infty)$. Since
$\C$ is Archimedean, there exists a $\lambda >0$ such that $(\lambda\Unit- c)\in \C$. Set $\epsilon'\coloneqq\epsilon \lambda^{-1}$. Since  $a+\epsilon'c\in \C$, we have
$ a-\epsilon \Unit=(a+\epsilon' c ) + \epsilon'(\lambda \Unit- c)\in \C$, that is, $a$ belongs to the set in \eqref{dagger1}. The converse inclusion is trivial (take $c=1$).
\end{proof}

For preorderings  $\C$ of polynomial algebras the $^\dagger$-operation was introduced by S. Kuhlmann and M. Marshall \cite{kumarshall},
using a slightly different notation, and studied in \cite{kumarshall, kumsch, cmn, netzer}.

Clearly, $C^\dagger$ is again a convex cone in $\A$.
Note that because of $\Unit\in\C$ we have $\C\subseteq \C^\dagger$.
As shown in \cite[Proposition 1.1]{cmn}, $\C^\dagger$ is the sequential closure of $\C$ in the finest locally convex topology on the vector space $\C-\C$. 
 
If $\Q$ is a quadratic module of $\A$, then $\A=\Q-\Q$. Moreover,  if $\{a_1,\dots,a_r\}$ are generators of the algebra $\A$ and 
$\big(\lambda \Unit -\sum_{i=1}^r a_i^2\big)\in \Q$ for some $ \lambda >0$, then $\Q$ is Archimedean. Both facts are no longer
valid for semirings. Proposition \ref{semiringgen} below is a counterpart for the latter assertion in the case of semirings. General
speaking, semirings can be rather ''small" (for instance, $\S=\RR_+\cdot \Unit$ is a semiring) and they are  more subtle to deal with. 
 
If $\S$ is a semiring, it is easily verified that $\S-\S$ is a unital subalgebra of $\A$.
\begin{definition}
  A semiring is called \emph{generating} if\, $\A=\S-\S$.
\end{definition} 
An Archimedean semiring is always generating, since $a=\lambda \Unit -(\lambda \Unit -a)$ for $a\in \A$ and $\lambda \in \RR$.

\begin{proposition} \label{proposition:sdaggersemiring}
  If $\S$ is a generating semiring of $\A$ and $\C$ an $\S$-module, then $\S^\dagger$ is again a generating semiring of $\A$ and $\C^\dagger$ is an $\S^\dagger$-module.
\end{proposition}
\begin{proof}
  It is clear that $\C^\dagger$ and $\S^\dagger$ are convex cones of $\A$  containing $\Unit$.
  Now let $a_1 \in \C^\dagger$ and $a_2\in \S^\dagger$ be given. Then there exist
  $b_1 \in \C$ and $b_2 \in \S$ such that $a_1 + \delta b_1 \in \C$ and $a_2 + \delta b_2 \in \S$ for all $\delta \in (0,\infty)$,
  and there exist $c_1 \in \C$ and $c_2 \in \S$ such that $c_1 - a_1 \in \C$ and $c_2 - a_2 \in \S$ because $\C - \C \supseteq \S-\S = \A$.
  Then $d_1 \coloneqq b_1+c_1 \in \C$ and $d_2 \coloneqq b_2 + c_2 \in \S$ fulfill $a_1 + \delta d_1 \in \C$ and $a_2 + \delta d_2\in \S$ for all $\delta \in (0,\infty)$,
  and $d_1 - a_1 \in \C$, $d_2 - a_2 \in \S$. Given  $\epsilon \in {(0,\infty)}$, we set 
  $\delta \coloneqq {-1 + \sqrt{1+\epsilon}}$. Then $\delta > 0$ and $\delta^2 + 2\delta = \epsilon$, and
  \begin{equation*}
    a_1a_2 + \epsilon d_1 d_2
    =
    (a_1 + \delta d_1)(a_2 + \delta d_2) + \delta\big((d_1-a_1) d_2 + d_1(d_2-a_2) \big)
    \in
    \C
    .
  \end{equation*}
  This shows that $a_1 a_2 \in \C^\dagger$. In particular, the special case $\C = \S$ shows that $\S^\dagger$ is a semiring,
  and $\S^\dagger$  is also generating because $\S \subseteq \S^\dagger$. In the general case, we have shown that $\C^\dagger$  is an $\S^\dagger$-module.
\end{proof}

\smallskip
 
We close this section  by recalling three results which are needed later. They are stated in  convenient forms for our purposes.

The first is the {\it Archimedean Positivstellensatz for quadratic modules}. For finitely generated quadratic modules of\,  $\RR[x_1,\dots,x_d]$ it  is usually called {\it Putinar's theorem} \cite{put94} in the literature.
\begin{proposition}\label{archquad}
  Suppose that $\Q$ is an Archimedean quadratic module of $\A$ and let $a\in \A$. If $\varphi(a)>0$ for all $\varphi\in \characters(\A,\Q)$, then $a\in \Q$.
\end{proposition}
\begin{proof}
  \cite[Theorem 4]{jacobi} or \cite[Theorem 5.4.4]{marshall} or \cite[Theorem 12.35]{sch17}.
\end{proof}
 
The next result is the {\it Archimedean Positivstellensatz for preorderings} \cite{sch91}.
\begin{proposition}\label{archpreorder}
  Let $\A\coloneqq\RR[x_1,\dots,x_d]$ and $f_1,\dots,f_r\in \A$, $r\in\NN$.  Recall that $\T(f_1,\dots,f_r)$ denotes
  the preordering \eqref{preorderingf}. Suppose that the semialgebraic set $\K(f_1,\dots,f_r)$ defined by \eqref{semiset}
  is compact. Let $p\in \RR[x_1,\dots,x_d]$. If $p(x)>0$ for all $x\in \K(f_1,\dots,f_r)$, then $p\in \T(f_1,\dots,f_r)$.
 \end{proposition}
\begin{proof}
  \cite[Theorem 19]{sch09}   or \cite[Corollary 6.1.2]{marshall}  or \cite[Theorem 12.24]{sch17}.
\end{proof}
 
The following proposition is {\it Havilands's theorem} \cite{Haviland} stated in a more general form.
\begin{proposition}\label{haviland}
  Suppose $\A$ is a finitely generated unital $\RR$-algebra and $\K$   a closed subset of $\characters(\A)$. 
  For any linear functional $L$ on  $\A$, the following are equivalent:
  \begin{enumerate}
    \item There exists a Radon measure $\mu$ supported on  $\K$ such that\, $L(a)=\int_\K a(x) \,\D\mu(x)$\, for all $a\in\A$.
      In this case, $L$ is called a \emph{$\K$-moment functional}.
    \item $L(a)\geq 0$ for all $a\in \Pos(\K)$.
    \item \label{haviland:ddagger} For any $a\in \Pos(\K)$ there is $b\in \Pos(\K)$ such that $L(f+\varepsilon b) >0$ for all $\varepsilon>0$.
  \end{enumerate}
\end{proposition}
\begin{proof}
  See e.g. \cite[Theorem 1.14]{sch17}.
\end{proof}

\section{Archimedean Semirings}\label{archsemirings}
We begin with a technical lemma.
\begin{lemma} \label{lemma:squareapprox}
  For all $k\in \NN, k\geq 2$,  the following identity of polynomials in one variable  holds:
  \begin{equation}\label{polidentity}
    \frac{1}{2^k k (k-1)}
    \sum_{\ell = 0}^{k} \binom{k}{\ell} (k-2\ell)^2 (1+x)^{k-\ell}(1-x)^{\ell} 
    =
    x^2 + \frac{1}{k-1}.
  \end{equation}
\end{lemma}
\begin{proof}
  {
  \allowdisplaybreaks
  As $(k-2\ell)^2 = k^2 - 4(k-1)\ell + 4\ell(\ell-1)$ for all $\ell\in \NN_0$, using the binomial theorem we compute
  \begin{align*}
    &\quad\quad\sum_{\ell = 0}^{k} \binom{k}{\ell} (k-2\ell)^2 (1+x)^{k-\ell}(1-x)^{\ell} 
    =
    \\
    &=
    2^k k^2
    -4(k-1)
    \sum_{\ell = 0}^{k} \binom{k}{\ell} \ell (1+x)^{k-\ell}(1-x)^{\ell} 
    +4
    \sum_{\ell = 0}^{k} \binom{k}{\ell} \ell(\ell-1)(1+x)^{k-\ell}(1-x)^{\ell} 
    \\
    &=
    2^kk^2
    -4k(k-1) \Bigg(
      \sum_{\ell = 1}^{k} \binom{k-1}{\ell-1} (1+x)^{k-\ell}(1-x)^{\ell} 
      -
      \sum_{\ell = 2}^{k} \binom{k-2}{\ell-2} (1+x)^{k-\ell}(1-x)^{\ell}
    \Bigg)
    \\
    &=
    2^kk^2
    -4k(k-1) \Bigg(
      \sum_{\ell = 0}^{k-1} \binom{k-1}{\ell} (1+x)^{k-1-\ell}(1-x)^{\ell+1} 
      -
      \sum_{\ell = 0}^{k-2} \binom{k-2}{\ell} (1+x)^{k-2-\ell}(1-x)^{\ell+2}
    \Bigg)
    \\
    &=
    2^kk^2 - 4k(k-1) \big( 2^{k-1}(1-x) - 2^{k-2} (1-x)^2 \big)
    \\
    &=
    2^kk^2 - 4k(k-1) \big( 2^{k-2} - 2^{k-2} x^2\big)
    \\
    &=
    2^k k + 2^k k (k-1) x^2
  \end{align*}
  which implies the identity \eqref{polidentity}.
  }
\end{proof}

The crucial fact for our approach is the following theorem.
\begin{theorem}\label{archpre}
  Suppose that $\S$ is an Archimedean semiring and $\C$ is an $\S$-module. Then  $\C^\dagger$ is an Archimedean quadratic module,
  and $\S^\dagger$ is a preordering.
\end{theorem}
\begin{proof}
  By Proposition~\ref{proposition:sdaggersemiring}, $\S^\dagger$ is a semiring and $\C^\dagger$ an $\S^\dagger$-module.
  Moreover, $\C^\dagger$ is Archimedean because $\S \subseteq \C \subseteq \C^\dagger$ and because $\S$ is Archimedean by assumption.
  
  It only remains to show that $a^2 \in \S^\dagger$ for all $a\in \A$. Since the semiring $\S$ is Archimedean,
  there exists a number $\lambda >0$  such that $( \lambda \Unit + a)\in \S$ and $(\lambda \Unit - a)\in \S$. Then 
  $(\Unit + a / \lambda)\in \S$ and $(\Unit - a/\lambda)\in \S$ and hence $(\Unit + a / \lambda)^n\in \S$ and 
  $(\Unit - a/\lambda)^n\in \S$ for all $n\in \NN_0$, because $\S$ is a semiring. Here, as usual, we 
  set $(\Unit \pm a/\lambda)^0\coloneqq\Unit$. Using Lemma~\ref{lemma:squareapprox} and the fact that $\S$ is closed under multiplication, one finds
  \begin{equation*}
    (a/\lambda)^2 + \frac{1}{k-1} \Unit
    =
    \frac{1}{2^k k (k-1)}
    \sum_{\ell = 0}^{k} \binom{k}{\ell} (k-2l)^2 \big(\Unit+(a/\lambda)\big)^{k-\ell}\big(\Unit-(a/\lambda)\big)^{\ell}
    \in
    \S
  \end{equation*}
  for $k\in \NN$, $k\geq 2$. Thus, $(a^2+\frac{\lambda ^2}{k-1}\Unit )\in \S$ for all $k\in \NN, k\geq 2$. This implies that
  $(a^2 + \epsilon \Unit) \in \S$ for all $\epsilon >0$. Therefore, $a^2 \in \S^\dagger$.
\end{proof}  
 
The following important result is a version of the {\it Archimedean Positivstellensatz for semirings}. In the case $\S=\C$
it was discovered by J.L. Krivine \cite{krivine}. A very  general version for $\S$-modules was obtained by  
T. Jacobi \cite{jacobi}. 
\begin{theorem}\label{archpos}
  Suppose that  $\S$ is an Archimedean semiring and $\C$ is an $\S$-module of the commutative unital real algebra $\A$.
  For any $a\in \A$, the following are equivalent:
  \begin{enumerate}
    \item \label{item:archpos:ana} $\varphi(a)>0$ for all $\varphi\in\characters(\A,\C)$.
    \item \label{item:archpos:alg} There exists $\epsilon \in {(0,\infty)}$ such that $a \in \epsilon \Unit + \C$.
  \end{enumerate}
\end{theorem}
\begin{proof} 
  \refitem{item:archpos:alg}$\implies$\refitem{item:archpos:ana}:
  Suppose that $a=\epsilon \Unit +c$ with $c\in \C$. Then, for $\varphi\in \characters(\A,\C)$, 
  we have $\varphi(a)=\epsilon\varphi(\Unit)+\varphi(c)= \epsilon+\varphi(c)\geq \epsilon >0$.

  \refitem{item:archpos:ana}$\implies$\refitem{item:archpos:alg}:
  Suppose that \refitem{item:archpos:ana} is satisfied.
  As noted in Section \ref{basicdef}, $ \characters(\A,\C)$ is compact in the weak topology and the function
  $\varphi\mapsto  \varphi(b)$ is continuous on $\characters(\A,\C)$ for any $b\in\A$. Hence there is an $\epsilon >0$
  such that $c\coloneqq a-\epsilon \Unit$ satisfies $\varphi(c)>0$ for all $\varphi\in \characters(\A,\C)$. Since 
  $\C\subseteq \C^\dagger$ and hence $\characters(\A,\C^\dagger)\subseteq \characters(\A,\C)$, we have 
  $\varphi(c)>0$ for $\varphi\in \characters(\A,\C^\dagger)$. By Theorem \ref{archpre}, $C^\dagger$ is 
  an Archimedean quadratic module.
  So the Positivstellensatz for Archimedean quadratic modules (Proposition \ref{archquad}) applies to 
  $c$ and $\C^\dagger$ and yields $c\in \C^\dagger$. Then  $a=c+\epsilon \Unit\in \C$ by the definition of $\C^\dagger$. 
\end{proof}
In the above proof, Theorem \ref{archpre} was used to derive  the Archimedean Positivstellensatz for semirings from the corresponding result for quadratic modules.
Thus Theorem \ref{archpre} allows one to develop a unified operator-theoretic approach to the Archimedean Positivstellensatz for semirings {\it and} quadratic modules at the same time. Indeed, the  Hilbert space proof given on 
\cite[p.~306]{sch17} applies  to an Archimedean quadratic module $\Q$ and to the  Archimedean quadratic module $\S^\dagger$ associated with an Archimedean semiring $\S$. We do not repeat the details  and refer to \cite{sch17}.

From the main implication \refitem{item:archpos:ana}$\implies$\refitem{item:archpos:alg} of Theorem \ref{archpos} 
it follows in particular that any $a\in \A$ satisfying \refitem{item:archpos:ana} belongs to the $\S$-module $\C$.
This is the  formulation of the Archimedean Positivstellensatz that appears often in the literature. The requirement
$a\in \varepsilon \Unit +\C$ in \refitem{item:archpos:alg} is stronger than $a\in \C$. It has the advantange 
that it gives an equivalent condition to \refitem{item:archpos:ana}.

\begin{corollary} \label{corollary:nonstrictArchimedean}
  If $\S$ is an Archimedean semiring on $\A$ and $\C$ an $\S$-module, then
  \begin{align}\label{cdaggerchar}
    \C^\dagger
    =
    \set[\big]{
      a\in \A
    }{ 
      \varphi(a)\geq 0 ~ \textup{ for all } ~\varphi\in \characters(\A,\C)
    }
  \end{align}
  and $\C^\dagger$ is a preordering.
\end{corollary}
\begin{proof}
  First suppose $\varphi(a)\geq 0$ for $\varphi\in \characters(\A,\C)$. Then, for each $\epsilon>0$,~
  $\varphi(a+\epsilon \Unit)=\varphi(a)+\epsilon >0$ for $\varphi\in \characters(\A,\C)$ and therefore
  $(a+\epsilon \Unit)\in \C$  by Theorem \ref{archpos}, so that $a\in \C^\dagger $. 

  Conversely, if $a\in \C^\dagger$ and $\varphi\in \characters(\A,\C)$, then $(a+\epsilon \Unit)\in \C$ 
  and hence $\varphi(a)+\epsilon=\varphi(a+\epsilon \Unit) \geq 0$  for all $\epsilon>0$. Letting 
  $\epsilon\searrow 0$, we get $\varphi(a)\geq 0$.   
  
  To prove that $\C^\dagger$ is a preordering, it  remains to show  that $\C^\dagger$ is closed under 
  multiplication. Indeed, if $a,b\in \C^\dagger$, then  \eqref{cdaggerchar} implies  
   $\varphi(ab)=\varphi(a)\varphi(b)\geq 0$ for $\varphi\in \characters(\A,\C)$, so  $ab\in \C^\dagger$ 
 by \eqref{cdaggerchar}. 
\end{proof}
 
That $\C^\dagger$  is closed under multiplication was derived in Corollary~\ref{corollary:nonstrictArchimedean}
from \eqref{cdaggerchar} and  Theorem \ref{archpos}. Now we prove this fact  by using only elementary algebraic arguments,
without invoking any Positivstellensatz.
\begin{proposition} \label{proposition:qmddagger}
  If $\Q$ is an Archimedean quadratic module of $\A$, then $\Q^\dagger$ is an Archimedean preordering of $\A$.
\end{proposition}
\begin{proof}
  It is clear that $\Q^\dagger$ is a convex cone of $\A$ that contains all squares. We only have to show that
  $\Q^\dagger$ is closed under multiplication.
  
  Let $p,q\in \Q$ and $\epsilon \in {(0,\infty)}$ be given. We prove that $pq + \epsilon \Unit \in \Q$.
  There exists $\lambda \in {(0,\infty)}$ such that $\lambda \Unit - p \in \Q$.
  We recursively define a sequence $(r_k)_{k\in \NN_0}$ in $\A$ by $r_0 \coloneqq p/\lambda$
  and $r_{k+1} \coloneqq 2 r_k - (r_k)^2$ for all $k\in \NN_0$. This sequence has the  property that
  $pq - 2^{-k}\lambda q r_k \in \Q$ for $k\in \NN_0$. By adding $2^{-(k+1)}\lambda (2 q r_k + q^2 + r_k^2) \in \Q$
  this implies
  $pq + 2^{-(k+1)}\lambda (q^2 + (r_k)^2) \in \Q$ for $k\in \NN_0$. Further, using
  \begin{equation*}
    r_{k+1} = \frac{(r_k)^2(2\Unit-r_k) + (2\Unit-r_k)^2r_k}{2} \in \Q
    \quad\quad\textup{and}\quad\quad
    \Unit - r_{k+1} = (\Unit-r_k)^2 \in \Q
  \end{equation*}
  one inductively finds that $r_k, \Unit-r_k \in \Q$, hence\, $2\cdot\Unit - (r_k)^2 = r_{k+1} + 2(\Unit-r_k) \in \Q$.
  For sufficiently large $k\in \NN_0$ it follows that\, $\epsilon \Unit - 2^{-(k+1)}\lambda (q^2 + (r_k)^2) \in \Q$
  because $\Q$ is Archimedean. Adding\,
  $pq + 2^{-(k+1)}\lambda (q^2 + (r_k)^2) \in \Q$ as above yields $pq+\epsilon\Unit \in\Q$.
  
  Now let $r,s \in \Q^\dagger$ and  $\epsilon \in {(0,\infty)}$.
  As $\Q$ is Archimedean, there exists $\lambda \in {(0,\infty)}$ such that $\lambda \Unit - (r+s) \in \Q$.
  Set $\delta \coloneqq \sqrt{\lambda^2 + \epsilon} - \lambda \in {(0,\infty)}$ and note that $\delta^2 + 2\lambda \delta = \epsilon$.
  Then $r+\delta \Unit, s+\delta \Unit \in \Q$, and therefore
  \begin{equation*}
    rs + \epsilon \Unit = (r+\delta\Unit)(s+\delta \Unit) + \delta \lambda \Unit + \delta\big(\lambda \Unit - (r+s) \big) \in \Q .
  \end{equation*}
  This shows that $rs\in \Q^\dagger$.
\end{proof}

\begin{corollary}
  Let $\S$ be an Archimedean semiring on $\A$ and $\C$ an $\S$-module, then $\C^\dagger$ is an Archimedean preordering of $\A$.
\end{corollary}
\begin{proof}
 Since  $\C^\dagger$ is an Archimedean quadratic module by Theorem~\ref{archpre}, 
   $(\C^\dagger)^\dagger$ is an Archimedean preordering by the previous Proposition~\ref{proposition:qmddagger}.
But $(\C^\dagger)^\dagger = \C^\dagger$, because $\C^\dagger$ is  Archimedean.
\end{proof}
Another possible application of Proposition~\ref{proposition:qmddagger} is to reduce the Positivstellensatz for Archimedean quadratic modules
to the analogous result for Archimedean semirings. Thus, the Archimedean
Positivstellens\"atze for semirings and quadratic modules  are equivalent by purely algebraic considerations as in Theorem~\ref{archpre} and Proposition~\ref{proposition:qmddagger}.

Finally, let us mention another aspect. It should be emphasized that the Archimedean Positivstellensatz and the corresponding 
proofs given in \cite{jacobi}, \cite{marshall} and \cite{sch17} are valid for arbitrary commutative unital real algebras
and  arbitrary Archimedean quadratic modules resp. semirings. In contrast to the Positivstellensatz for preorderings proved
in \cite{sch91}, it is not needed to assume that the algebra, the quadratic module or the semiring are finitely generated.
The proofs in \cite{sch91}, \cite{sch09} are essentially based on the Stengle-Krivine Positivstellensatz which requires a finitely generated preordering.

\section{Sharpening the  Archimedean Positivstellensatz for Preorderings}\label{sharpening}
We begin with an example which illustrates the results developed below. \begin{example}\label{module}
  Let $\S$ denote the semiring of $\A=\RR[x_1,\dots,x_d]$ generating by the $(2d+1)$ polynomials
  \begin{align}
    f(x) \coloneqq 1-x_1^2-\cdots-x_d^2,~
    g_{j,\pm}(x) \coloneqq (1\pm x_j)^2,\;
    j=1,\dots,d.
  \end{align}
  Clearly,  $\characters(\A,\S)$ is the set of point evaluation functionals at the points of the closed unit ball
  \begin{equation*}\K(f)=\set{x\in \RR^d}{ x_1^2+\cdots+x_d^2\leq 1}.
  \end{equation*}
  Then, since 
  \begin{equation*}
    d+1\pm 2x_k=(1-x_1^2-\cdots-x_d^2) +(1\pm x_k)^2+\frac{1}{2} \sum_{i=1, i\neq k}^d \big((1+x_j)^2+(1-x_j)^2\big)\in \S,
  \end{equation*}
   $x_k\in \A^\bd(\S)$\, for  $k=1,\dots,d$. Hence the semiring $\S$ is Archimedean. Therefore, by Theorem
  \ref{archpos}, each polynomial $p\in \RR[x_1,\dots,x_d]$ that is positive in all points of  the closed unit ball $\K(f)$ is of the form
  \begin{align*}
    &p(x)=\sum_{n,k_i,\ell_i=0}^m \alpha_{n,k_1,\ell_1,\dots,k_d,\ell_d}~(1{-}x_1^2{-}\cdots{-}x_d^2)^{2n}(1-x_1)^{2k_1}(1+x_1)^{2\ell_1}\cdots (1-x_d)^{2k_d}(1+x_d)^{2\ell_d} ~ +\\&(1{-}x_1^2{-}\cdots{-}x_d^2)\sum_{n,k_i,\ell_i=0}^m \beta_{n,k_1,\ell_1,\dots,k_d,\ell_d} (1{-}x_1^2{-}\cdots{-}x_d^2)^{2n}(1-x_1)^{2k_1}(1+x_1)^{2\ell_1}\cdots (1-x_d)^{2k_d}(1+x_d)^{2\ell_d}
 \end{align*} 
 with\, $m\in \NN_0$ and where $\alpha_{n,k_1,\ell_1,\dots,k_d,\ell_d}$, $\beta_{n,k_1,\ell_1,\dots,k_d,\ell_d}$ are nonnegative real numbers. 
 This formula represents the positive polynomial $p$ on $\K(f)$ in terms of a  subset of the preordering $\T(f)$  given by some distinguished weighted squares of polynomials. 
  \end{example}

The following considerations generalize the preceding example.
Let $g_1,\dots,g_m$, $m\in \NN$,  be a fixed set of polynomials of $ \A \coloneqq \RR[x_1,\dots,x_d]$ which {\it generates the algebra}
$\A$. Further, let $f_1,\dots,f_r\in \RR[x_1,\dots,x_d]$, $r\in \NN$, and suppose   that the {\it semialgebraic set 
  \begin{equation}\label{semikf}
   \K(f_1,\dots,f_r)\coloneqq \set[\big]{ x\in \RR^d }{ f_1(x)\geq 0,\dots,f_r(x)\geq 0 }
  \end{equation}
  is compact. }
  
  Since the set $\K(f_1,\dots,f_r)$ is compact,  there exist numbers $\alpha_j>0, \beta_j>0$ such that 
  \begin{equation}\label{poshj}
    \alpha_j+g_j(x)>0~~\textup{and}~~ \beta_j-g_j(x)>0~~ \textup{for}~~ x\in \K(f_1,\dots,f_r),~j=1,\dots,m.
  \end{equation} 
  Recall that $\T(f_1,\dots,f_r)$ is the preordering  of $\A$ generated by $f_1,\dots,f_r$.
  The Archimedean Positivstellensatz for preorderings (Proposition \ref{archpreorder})
   implies that 
  \begin{equation}\label{pospre}
    \alpha_1+ g_1(x),\beta_1-g_1(x),\dots,\alpha_m+g_m(x),\beta_m-g_m(x)\in \T(f_1,\dots,f_r)
    .
  \end{equation} 
  By the definition \eqref{preorderingf} of the preordering $\T(f_1,\dots,f_r)$, \eqref{pospre}
  means that each polynomial $\alpha_j+ g_j$ and  $\beta_j-g_j$ is a finite sum of  polynomials
  of the form $f_1^{e_1}\cdots f_r^{e_r}p^2$ with $p\in \A$ and $e_1,\dots,e_r\in \{0,1\}$. Let $\S$
  denote the semiring generated by the polynomials $f_1,\dots,f_r$ and all squares $p^2$ occurring
  in the corresponding representations \eqref{pospre} of the   polynomials 
  $\alpha_1+ g_1,\beta_1-g_1,\dots,\alpha_m-g_m,\beta_m-g_m$. These representations imply that  $g_1,\dots,g_m$ belong to $\A^\bd(\S)$. Therefore, since the set $g_1,\dots,g_m$ generates the algebra  $\A$ by assumption, it follows from Lemma~\ref{archhilfs} that
  the semiring $\S$ is Archimedean.
  Finally, note that $\characters(\A,\S)$ is the set of point evaluations at $\K(f_1,\dots,f_r)$ because $f_1,\dots,f_r\in\S$.

\begin{theorem}\label{auxsemiring}
  Suppose that $\{f_1,\dots,f_r\}$, $r\in \NN$, is a subset of $\A=\RR[x_1,\dots,x_d]$ such that the semi\-algebraic set 
  $\K(f_1,\dots,f_r)$
  is compact.   
  Then there  exist polynomials $p_1,\dots,p_s\in \RR[x_1,\dots,x_d], s\in \NN,$ such that
  the semiring $\S$ generated by $f_1,\dots,f_r, p_1^2, \dots,p_s^2$ is Archimedean. 

  If $q\in \RR[x_1,\dots,x_d]$
  satisfies $q(x)>0$ for all  $x\in \K(f_1,\dots,f_n)$, then $q$ is a finite sum of polynomials 
  \begin{align}\label{specialform}
    \alpha\, f_1^{e_1}\cdots f_r^{e_r}~ f_1^{2n_1}\cdots f_r^{2n_r}\, p_1^{2k_1}\cdots p_s^{2k_s},
  \end{align}
  where $\alpha\in (0,+\infty)$, $e_1,\dots,e_r\in \{0,1\}$, $n_1,\dots,n_r, k_1,\dots,k_s\in \NN_0$.
 
  Further, each linear functional on $\RR[x_1,\dots,x_d]$ that is nonnegative on all polynomials \eqref{specialform}
  (with $\alpha=1$) is a $\K(f_1,\dots,f_r)$-moment functional.
\end{theorem}
\begin{proof}
  By its construction, the semiring $\S$ defined above is generated by polynomials $f_1,\dots,f_r$,  $p_1^2,\dots,p_s^2$. Since $\S$  is Archimedean, the Archimedean Positivstellensatz
  for semirings (Theorem \ref{archpos})   yields $q\in \S$. This means that $q$ is  a finite sum of
  terms \eqref{specialform}. 
  By Haviland's theorem (Proposition \ref{haviland}) this implies the last assertion.
\end{proof}

Let us briefly discuss the preceding construction and results:
Recall that by the Archimedean Positivstellensatz for
preorderings, positive polynomials  on the compact semialgebraic set $\K(f_1,\dots,f_r)$ are finite sums of terms $f_1^{e_1}\cdots f_r^{e_r} \sigma$,
with $\sigma\in \sum \RR[x_1,\dots,x_d]^2$.
Theorem \ref{auxsemiring} sharpens this assertion: It specifies the sums of squares  and shows that  the preordering $\T(f_1,\dots,f_r)$ can be replaced by a finitely generated subsemiring.
Note that  a finitely generated preordering  is not necessarily finitely generated as a semiring. 

It should be emphasized that the above construction applies to {\it any} set of generators $g_1,\dots,g_m$ of the algebra $\RR[x_1,\dots,x_d].$ For instance, we may choose $d$ linearly
independent linear polynomials $g_1,\dots,g_d$ satisfying $g_1(0)=\dots=g_d(0)=0$ and set $g_{d+1}=1$. Then \eqref{poshj}
says that semialgebraic set $\K(f_1,\dots,f_r)$ is strictly contained in the compact polyhedron
\begin{align*}
  \set[\big]{ x\in \RR^d }{ \beta_1\geq g_1(x)\geq -\alpha_1,\dots, \beta_d\geq g_d(x)\geq -\alpha_d }
  .
\end{align*}

\section{Archimedean Semirings Generated by Supporting Polynomials of Compact Convex Sets}\label{convexsets}

Let $\A_d$ denote the real vector space of polynomials in $x_1,\dots,x_d$ of degree at most one, i.e.
\begin{equation*}
  \A_d = \set[\Big]{\alpha + \sum\nolimits_{j=1}^d \beta_j x_j}{\alpha, \beta_1, \dots, \beta_d \in \RR} \subseteq \RR[x_1,\dots,x_d]
  .
\end{equation*}
For a subset $G$ of $\A_d$, let $\C(G)$ be the convex cone of $\A_n$  generated by $G \cup \{\Unit\}$ and we define
\begin{equation}\label{eqdefinitionkg}
  \K(G) \coloneqq \set[\big]{t\in \RR^d}{g(t) \ge 0~~ \textrm{for all}~~ g\in G}
  .
\end{equation}
Clearly, the set $\K(G)$ is convex and closed.
We will use the following version of the Farkas lemma.

\begin{lemma} \label{lemma:affineSeparation}
  Let $G$ be a subset of $\A_n$. Suppose that $\K(G)$ is non-empty and compact. Then, for every $h\in \A_d$  such that 
  $h\notin  \C(G)$ there exists a point $t_0\in \K(G)$ such that $h(t_0) \le 0$.
\end{lemma}
\begin{proof}
  The vector space $\A_d$ has dimension $d+1$. Its dual space $\A_d^*$ is given by the functionals 
  \begin{equation*}
    \omega_{\lambda,t}(f)\coloneqq \lambda\alpha+\sum\nolimits_{j=1}^d \beta_jt_j~~ \textrm{for}~~ f= \alpha+\sum\nolimits_{j=1}^d \beta_jx_j \in \A_d,~~\textrm{where}~ \lambda\in \RR,~t=(t_1,\dots,t_d)\in \RR^d.
  \end{equation*}
  Clearly, 
  \begin{equation*}
     \omega_{\lambda,t}(f) \coloneqq f(t) + (\lambda-1) f(0).
  \end{equation*}
  Since  $\RR^d \ni t \mapsto f(t)-f(0) \in \RR$ is linear for $f\in \A_d$, it follows that $\RR^{1+d} \ni (\lambda,t) \mapsto \omega_{\lambda,t}\in \A_d^*$ is a linear map.
   Note that $\omega_{1,t}$ is just the evaluation functional at $t$.
  
  Since $h$ does not belong to the convex set $\C(G)$, by the convex separation theorem there exist a linear functional
  $\omega_{\lambda,t}\neq 0$, where $\lambda \in \RR$, $t\in \RR^d$, such that
  $\omega_{\lambda,t}(h) \le \omega_{\lambda,t}(f)$ for all $f\in \C(G)$.  From $0\in \C(G)$ it follows that $\omega_{\lambda,t}(h) \le 0$.
  If $f\in \C(G)$, then also $\mu f\in \C(G)$, hence $\omega_{\lambda,t}(h) \le \omega_{\lambda,t}(\mu f) = \mu \,\omega_{\lambda,t}(f)$, for all $\mu \in {[0,\infty)}$,
  so we have $\omega_{\lambda,t}(h) \le 0 \le \omega_{\lambda,t}(f)$. From $\Unit \in \C(G)$ we get $0 \le \omega_{\lambda,t}(\Unit) = \lambda$.
  
  We verify that $\lambda \neq 0$.
   Assume to the contrary that $\lambda = 0$, i.e.~$\omega_{0,t}(h) \le 0 \le \omega_{0,t}(f)$ for all $f\in \C(G)$.
  By assumption $\K(G)$ is not empty, so there exists a point $s\in \K(G)$, so $0 \le g(s) = \omega_{1,s}(g)$ for all $g\in \C(G)$.
  But then $0 \le \omega_{1,s}(g) + \lambda \omega_{0,t}(g) = \omega_{1,s+\lambda t}(g) = g(s+\lambda t)$ for all 
  $\lambda \in {[0,\infty)}$ and $g\in G$, which means $s+\lambda t \in \K(G)$ for all $\lambda \in {[0,\infty)}$.
  Since $t\neq 0$ because $\omega_{0,t}\neq 0$, we have a contradiction to the compactness of $\K(G)$. Thus, $\lambda >0$.
  
  Therefore,  for $f\in \C(G)$, we obtain
  \begin{equation*}
    \lambda\, h(\lambda^{-1} t) = \lambda\, \omega_{1,\lambda^{-1} t}(h) = \omega_{\lambda,t}(h) \le 0 \le \omega_{\lambda,t}(f) = \lambda \,\omega_{1,\lambda^{-1}t}(f) = \lambda\, f(\lambda^{-1} t).
  \end{equation*}
  In particular, $0 \le g(\lambda^{-1} t)$ for all $g\in G$, that is, ~$t_0\coloneqq\lambda^{-1} t \in \K(G)$, and $h(t_0) \le 0$.
\end{proof}

The assertion of Lemma \ref{lemma:affineSeparation} does not hold in general if $\K(G)$ is empty (for instance, take 
$G=\{x-k\Unit\, |\, k\in \NN\}\subseteq \A_1$ and $h=-\Unit$ like in Example~\ref{example:pathological} below).

\begin{proposition}\label{propokg}
  Let $G$ be a subset of $\A_d$ and let $\S(G)$ be the semiring of\, $\RR[x_1,\dots,x_d]$  generated by $G$. Suppose that the set  $\K(G)$ is non-empty and compact.
  Then $\S(G)$ is Archimedean and  each $p\in \RR[x_1,\dots,x_d]$  such that $p(t) > 0$ for all $t\in \K(G)$ is in $\S(G)$.
\end{proposition}
\begin{proof}
  Note that $\C(G)\subseteq \S(G)$ by definition.
  Since $\K(G)$ is compact, there exists $\lambda \in {(0,\infty)}$ such that $\lambda - t_j > 0$ and $\lambda + t_j > 0$ for all $t\in \K(G)$ and $j=1,\dots,d.$
 Then $\lambda\Unit - x_j, \lambda\Unit + x_j \in \C(G) \subseteq \S(G)$ by Lemma~\ref{lemma:affineSeparation}. Hence $\S(G)$ is Archimedean by Lemma~\ref{archhilfs} and  
 Theorem \ref{archpos}  applies.
\end{proof}
An application of Proposition~\ref{propokg} is the following.
A similar characterization of strictly positive elements is not true  for quadratic modules,
because they  contain squares that vanish at points of $C$.

\begin{corollary} \label{corollary:truestrictsatz}
  Let $C$ be a non-empty  compact convex subset of $\RR^d$ and set
  \begin{equation}
    G \coloneqq \set[\big]{g\in \A_d}{g(t) > 0~\textup{ for all }~t\in C}
    .
  \end{equation}
  Then for $p\in \RR[x_1,\dots,x_d]$ the following are equivalent:
  \begin{enumerate}
    \item\label{item:truestrictsatz:ana} $p(t) > 0$ for all $t\in C$.
    \item\label{item:truestrictsatz:alg} \label{item:archposechtstrikt:alg} $p$ is a  sum of products of elements of $G$.
  \end{enumerate}
\end{corollary}
\begin{proof}
  Note that $\S(G)\setminus\{0\}$ is the set of elements  in ~\refitem{item:truestrictsatz:alg}.
  From the  separation theorem for convex set it follows that $\K(G) = C$. Thus, Proposition~\ref{propokg} yields
  \refitem{item:truestrictsatz:ana} $\Longrightarrow$ \refitem{item:truestrictsatz:alg}.
  The converse implication is trivial.
\end{proof}

Finally, we will  restate Proposition \ref{propokg} in terms of convex analysis as Theorem \ref{proposition:compactconvexssemiring}.

Suppose $C$ is  a compact convex set in $\RR^d$.
By a \emph{supporting affine functional} at a point $t_0\in C$   we mean a polynomial $h\in \A_d$, $h\neq 0$,
such that $h(t_0)=0$ and $h(t)\geq 0$ for all $t\in C$. If such a functional exists, then $t_0$ is a boundary point of $C$. The set of supporting affine functionals at $t_0$ is denoted by $\cH(t_0;C)$.
By definition, $C$ is contained in the supporting half-space $\set{t\in \RR^d }{ h(t)\geq 0 }$ for each $h\in \cH(t_0;C)$.
 It is well-known (and follows easily from a separation theorem of convex sets) that
\begin{equation}\label{cintersectionh}
  C = \bigcap\nolimits_{t_0\in C} \set[\big]{ t\in \RR^d }{ h(t)\geq 0 ~~ \textup{for all}~~ h\in \cH(t_0;C) }.
\end{equation}
In general, not all points $t_0\in C$ and  functionals $h\in \cH(t_0;C)$ are needed to get the equality \eqref{cintersectionh}.

\begin{theorem} \label{proposition:compactconvexssemiring}
  Let $C$ be a non-empty compact convex subset of $\RR^d$, $d\in \NN$. Suppose that  $\cH$ is a set of supporting affine functionals at  points of $C$ such that
  \begin{equation}\label{conditioncH}
    C= \set[\big]{ t\in \RR^d }{ h(t)\geq 0 ~~\textup{for all}~~ h\in \cH }.
  \end{equation}
   Let $\S(\cH)$ denote the  semiring of\, $\RR[x_1,\dots,x_d]$ generated by  $\cH$.
  Then $\S(\cH)$ is  Archimedean. If $f\in \RR[x_1,\dots,x_d]$ satisfies $f(t)>0$ for all $t \in C$, then $f\in \S(\cH)$.
\end{theorem}
\begin{proof} 
Since $\K(\cH)=C$ by assumption \eqref{conditioncH} and definition \eqref{eqdefinitionkg}, the assertions follow at once from Proposition \ref{propokg}.
\end{proof}

\begin{example}
  Let $C \coloneqq \set{(x,y)\in\RR^2}{x^2 + y^2 \le 1}$ be the unit disc in $\RR^2$. For $\vartheta \in {[0,2\pi)}$ we define
  $h_\vartheta \coloneqq 1 + \cos(\vartheta) x + \sin(\vartheta) y$. Then the  set of {\it all} supporting affine functionals at points of $C$ is
  $\set{\lambda h_\vartheta}{\lambda \in {(0,\infty)}}$. Therefore, it follows from Proposition~\ref{proposition:compactconvexssemiring}
  that each polynomial $f\in \RR[x,y]$ such that $f(t)>0$ for all $t\in C$ belongs to the semiring generated by all $h_\vartheta$, 
  $\vartheta \in {[0,2\pi)}$. If $J$ is a dense subset of $[0,2\pi)$ and $\cH=\set{h_\vartheta }{ \vartheta \in J }$, then $f$ is even in $\S(\cH)$.
   
  Let $K(g_1,g_2) \coloneqq \set{ (x,y)\in \RR^2 }{ g_1 \coloneqq y-x+1\geq 0,~ g_2\coloneqq y+x+1\geq 0 }$.
  Then Proposition~\ref{proposition:compactconvexssemiring} applies  to $C_0\coloneqq C\cap\K(g_1,g_2)$  and 
  $\cH_0= \set{ g_1, g_2, h_\vartheta }{ \vartheta \in (0,\pi) }$ as well. Thus, if $f\in \RR[x,y]$ satisfies
  $f(t)>0$ for all $t\in C_0$, then  $f\in\S(\cH_0)$.
\end{example}

\section{Sharpening the Positivstellensatz for Semialgebraic Sets Contained in  Compact Polyhedra}\label{sharpenpolyhedra}

We begin with a general observation:
Let $\A=\RR[x_1,\dots,x_d]$ and $f_0 \coloneqq \Unit, f_1,\dots,f_r\in \A$, $r\in \NN$. Suppose that $\S$ is a semiring of
$\A$ and consider the $\S$-module $\C$ defined by
\begin{equation}\label{defcsm}
  \C \coloneqq f_0\cdot \S+ f_1\cdot \S+\cdots+f_r\cdot\S
  .
\end{equation}
Then it is easily verified that
\begin{equation}\label{setks}
  \characters(\A,\C) = \characters(\A,\S) \cap \set[\big]{ x\in \RR^d }{ f_1(x)\geq 0,\dots,f_r(x)\geq 0 }
  .
\end{equation}
Note that in general the $\S$-module $\C$ in \eqref{defcsm} is  neither a semiring nor a quadratic module. 
If the semiring $\S$ is finitely generated, then  $\characters(\A,\C)$  is  semialgebraic. However,  if $\S$ 
is not finitely generated,  $\characters(\A,\S)$ and   $\characters(\A,\C)$ are not necessarily semialgebraic. 

Suppose that $\S$ is Archimedean. Then Theorem \ref{archpos} gives the following: If $p\in \A$ and $p(t)>0$ for all
$t\in \characters(\A,\C)$, then $p\in \varepsilon\Unit+ \C$ for all $\varepsilon >0$.

\begin{theorem}\label{polyhedron2}
  Let $f_0=\Unit,f_1,\dots,f_r\in \A\coloneqq\RR[x_1,\dots,x_d]$, $r\in \NN,$ and let $G$ be a subset of $\A_d$ such that the convex set $\K(G)$
  is non-empty and compact. Recall that $\S(G)$ denotes the semiring of $\A$ generated by $G$.
  If $p\in\RR[x_1,\dots,x_d]$ satisfies $p(x)>0$ for all $x\in \K(f_1,\dots,f_r) \cap \K(G)$, then there are $h_0,\dots h_r \in \S(G)$
  such that
  \begin{equation}\label{containedinc}
    p = f_0 h_0 + \dots + f_r h_r.
  \end{equation}
\end{theorem}
\begin{proof}
  Consider the $\S(G)$-module $\C \coloneqq f_0 \cdot \S(G) + \dots + f_r \cdot \S(G)$. By \eqref{setks}, $\characters(\A,\C)$
  are the evaluation functionals at points of $\K(f_1,\dots,f_r) \cap \K(G)$.
  The semiring $\S(G)$ is Archimedean by Proposition~\ref{propokg}. Therefore the Positivstellensatz for
  modules of Archimedean semirings (Theorem \ref{archpos}) applies to the $\S(G)$-module $\C$:
  If $p(x)>0$ for all $x\in \K(f_1,\dots,f_r) \cap \K(G)$, then  $p \in \C$ which yields \eqref{containedinc}.
\end{proof}

\begin{corollary}\label{corollary:polyhedron1}
  Let $f_0=\Unit,f_1,\dots,f_r\in \RR[x_1,\dots,x_d]$, $r\in \NN$. Suppose $g_1,\dots,g_m\in\RR[x_1,\dots,x_d]$
  are polynomials of degree one such that the polyhedron\, $\K(g_1,\dots,g_m)\coloneqq \set[]{ x\in \RR^d }{ g_1(x),\dots,g_m(x)\geq 0 }$
  is a non-empty compact set that contains the semialgebraic set $\K(f_1,\dots,f_r)$ defined by \eqref{semikf}.
  If $p\in\RR[x_1,\dots,x_d]$ satisfies $p(x)>0$ for all $x\in \K(f_1,\dots,f_r)$, then $p$ is a finite sum of polynomials
  \begin{align}\label{specialform1}
    \alpha\, f_j \, g_1^{n_1}\cdots g_m^{n_m},~~~\textit{where}~~\alpha\in (0,+\infty),~ j=0,\dots,r;~n_1,\dots,n_m\in \NN_0.
  \end{align}
  Any linear functional on $\RR[x_1,\dots,x_d]$ that is nonnegative on all polynomials in \eqref{specialform1} (with $\alpha=1$)
  is a $\K(f_1,\dots,f_n)$-moment functional.
\end{corollary}
\begin{proof}  
  Consider the $\S$ module $\C$ defined by \eqref{defcsm}. By assumption,\, $\characters(\A,\S)= \K(g_1,\dots,g_m)$
  contains the set $\K(f_1,\dots,f_r)$. Hence $\characters(\A,\C)= \K(f_1,\dots,f_r)$ by \eqref{setks}.
  As noted in Example \ref{module},  Theorem \ref{archpos} implies that $p\in \C$. By \eqref{defcsm} this means that $p$
  is a sum of terms as in \eqref{specialform1}. The last assertion follows from  Proposition \ref{haviland}.
\end{proof}

Corollary \ref{corollary:polyhedron1} sharpens a well-known Positivstellensatz for semirings (see \cite[Theorem 5.4.6]{PD} or 
\cite[Theorem 12.44]{sch17}). Instead of mixed products $f_1^{k_1}\cdots f_r^{k_r}$ appearing therein 
there is only a single factor $f_j$  in \eqref{specialform1}.

\section{Further Examples and Ramifications of the Archimedean Positiv\-stellensatz for Semirings}\label{examplesarch}
In most applications of the Archimedean Positivstellensatz for semirings of polynomial algebras
one verifies  the Archimedean condition for the generators $x_i$. Here is another source for getting examples.

\begin{example} ({\it Archimedean semirings on polynomial algebras via automorphisms})\\
  We set
  \begin{equation*}
    f_1(x,y)\coloneqq y-(x-1/2)^2,~~
    f_2(x,y)\coloneqq y-x^2
  \end{equation*}
  and define a semialgebraic set of $\RR^2$ by 
  \begin{equation*}
    \K\coloneqq\set[\big]{ (x,y)\in \RR^2 }{ {-1}\leq f_1(x,y)\leq 1 ~\textup{ and }~ {-1}\leq f_2(x,y)\leq 1 }.
  \end{equation*}
  The set $\K$ is bounded by four parabolas and it is compact.

  Let $\S$ be the semiring generated by $1-f_1, 1+f_1, 1-f_2, 1+f_2$. Clearly, $\characters(\A,\S)=\K$.
  We show that $\S$ is Archimedean. By definition, $f_1, f_2 \in\A^\bd(\S)$. Hence $x=(f_1-f_2+1/4)\in \A^\bd(\S)$. Then,
  $x^2\in \A^\bd(\S)$ and $y=f_1+x^2 \in \A^\bd(\S)$.  Therefore, $\A^\bd(\S)=\A$, which proves that $\S$ is Archimedean.

  Thus the Archimedean Positivstellensatz  for semirings (Theorem \ref{archpos}) applies to $\S$ and $\K$.  It means that
  each polynomial $p\in \RR[x,y]$ that is strictly positive in all points of $\K$ is a finite sum of terms
  \begin{equation*}
     \alpha (1-f_1(x,y))^i (1+f_1(x,y))^j (1-f_2(x,y)\big)^k (1+f_2(x,y))^\ell,
  \end{equation*}
  where  $\alpha\geq 0$ and $i,j,k,\ell\in \NN_0$.

  In order to  generalize this example it is convenient to look at it from another perspective. Let $f(y)\in \RR[y], g(x)\in \RR[x]$
  be polynomials in one variable and $a,b,c,d\in \RR, ab-cd\neq 0$. 
  It is well known and easily verified that the following mappings $\varphi, \psi, \eta$ define automorphisms of the  algebra $\RR[x,y]$:
  \begin{equation}\label{autom}
    \varphi (x)\coloneqq x+f(y),~
    \varphi (y)\coloneqq y,~~
    \psi(x) \coloneqq x, ~
    \psi(y)\coloneqq y+g(x),~~
    \eta(x)\coloneqq ax+by, ~
    \eta(y)\coloneqq cx+dy.
  \end{equation}
  A classical result of H.W.E.~Jung \cite{jung} states that all such mappings $\varphi, \eta $ (and likewise $\psi,\eta$) generate the
  automorphism group of the polynomial algebra $\RR[x,y]$.
  Since each {\it composition} of  mappings of the form \eqref{autom} is an automorphism,   it maps $x,y$ to {\it generators}, say 
  $x'=f_1(x,y), y'=f_2(x,y),$ of the algebra $\RR[x,y]$. Therefore, if we fix numbers $\lambda_1>0, \lambda_2>0$ and let $\S$ be the
  semiring generated by $\lambda_1\pm f_1(x,y), \lambda_2\pm f_2(x,y)$,  then  $\S$ is Archimedean and the corresponding semialgebraic
  set  is compact, so the preceding facts remain valid almost verbatim. Choosing compositions of such automorphisms provides a large source
  of examples for which Theorem \ref{archpos} applies. The above example is obtained in the special case
  $x'=(\psi\circ\varphi)(x), y'=(\psi\circ\varphi)(y)$, where 
  \begin{equation*}
    \varphi (x)=x+y- 1/4,~~
    \varphi(y) =y,~~
    \psi(x)= x,~~
    \psi(y)=y-x^2.
  \end{equation*}
\end{example}
 
\begin{example}\label{circle}
  Let $\A=\RR[x,y]$ and let $\Gamma$ be a fixed  subset of $[0,2\pi)$. 
  Let $\S_\Gamma$ denote the semiring of $\A$ generated by 
  \begin{equation*}
    1-(x \cos \gamma+ y\sin \gamma)~~\textup{and}~~ 1+(x \cos \gamma+ y\sin \gamma),~~\textup{where}~~\gamma\in \Gamma.
  \end{equation*} 
  The semiring  $\S_\Gamma$ is not finitely generated if $\Gamma$ is infinite. 
  The set $\characters(\A,\S_\Gamma)$ consists of the point evaluations of 
  \begin{equation*}
    \K(\S_\Gamma)\coloneqq \set[\big]{ (x,y)\in \RR^2 }{ 1\geq  x \cos \gamma+ y\sin \gamma \geq -1\textup{ for all } \gamma\in \Gamma }.
  \end{equation*}
  For instance, if $\Gamma=\{0,\frac{\pi}{2} \}$, then $K(\S_\Gamma)$ is a square $[-1,1]\times [-1,1]$, and if $\Gamma=[0,2\pi)$, then 
  $\K(\S_\Gamma)$ is the unit circle in $\RR^2$.

  Suppose that $\S_\Gamma$  contains $0$ and $\frac{\pi}{2}$. Then $\Unit\pm x$ and $\Unit\pm y$ are in $\S_\Gamma$. Hence, by Lemma 
  \ref{archhilfs}, the semiring $\S_\Gamma$ is Archimedean and Theorem \ref{archpos} applies. 
\end{example}

\begin{example}
  Fix $\alpha\in(0,+\infty)$. Let $\A$ be the unital real algebra generated by the rational functions 
  \begin{equation}\label{defabc}
    a
    \coloneqq
    \frac{\alpha}{x^2+\alpha^2}\, ~~ {\rm and}~~
    b\coloneqq\frac{x}{x^2+\alpha^2}\,.
  \end{equation}
  Since $\A$ is unital, $\A$ is also generated by the functions $\frac{\gamma x^2+\delta x+\eta}{x^2+\alpha^2}$,
  where $\gamma,\delta,\eta\in \RR$. Note that 
  \begin{equation}\label{rela}
    \big(a-(2\alpha)^{-1}\big)^2+ b^2=(2\alpha)^{-2}.
  \end{equation}
  We choose a number $\beta>0$ such that 
  $\beta \alpha\geq (4\beta\alpha^3)^{-1}$. Let $\S$ denote the semiring of $\A$ generated by 
  \begin{align*}
    a,~~
     c \coloneqq \frac{ x^2}{x^2+\alpha^2},~~
    c_{\pm}\coloneqq \frac{(x\pm(2\beta)^{-1})^2}{x^2+\alpha^2}.
  \end{align*}
  
  We describe the set $\characters(\A,\S)$. Obviously, the point evaluation $\chi_x$  at $x\in \RR$ is in $\characters(\A,\S)$. Further,
  $\chi_\infty(f)=\lim_{t\to \infty} f(t), f\in \A$, yields a well-defined character $\chi_\infty\in\characters(\A,\S)$. 
  From \eqref{rela} it follows that $\chi\mapsto (\chi(a),\chi(b))$ defines a mapping $\Psi$ of $\characters(\A,\S)$ into  the circle 
  \begin{equation}\label{defK}
    K
    \coloneqq \set[\big]{ (s,t)\in \RR^2 }{ (s-(2\alpha)^{-1})^2+t^2=(2\alpha)^{-2} }.
  \end{equation}
  Since $\A$ is generated by $a$ and $ b$, $\Psi$ is injective. To show that $\Psi$ is surjective, let $(s,t)\in K$. First suppose $s\neq 0$. Set $x\coloneqq t\alpha s^{-1}$.
  Using  the definition \eqref{defK} of $K$ we compute $s=\chi_x(a)$, $t=\chi_x(b)$, so that $(s,t)=\Psi(\chi_x)$. If $s=0$,
  then $t=0$ by \eqref{defK} and hence $(s,t)=(0,0)=(\chi_\infty(a),\chi(b))$. This shows that  $\Psi$ is a  bijection of $\characters(\A,\S)$ and  $K$. In particular, this implies that 
  \begin{equation*}
    \characters(\A,\S) = \set[\big]{ \chi_t }{ t\in \RR\cup \{\infty\} }.
  \end{equation*}
  If $\lambda>0$ and $\lambda \alpha-1\geq 0$, then the elements
  \begin{equation*}
    \lambda \Unit \pm a=\lambda c+ (\lambda\pm1)a,~~
    \beta\Unit \pm b= \beta c_{\pm}+\big(\beta \alpha - (4\beta \alpha^3)^{-1}\big)a
  \end{equation*}
  belong to $\S$. Therefore, $\S$ is an Archimedean semiring and Theorem \ref{archpos} applies.
\end{example}

Another application of the Archimedean Positivstellensatz for semirings is the following result.
\begin{proposition}
  Let $\A_1 \coloneqq \RR[x_1,\dots,x_n]$ and $\A_2 \coloneqq \RR[x_{n+1},\dots,x_{n+k}]$ and let $\T_1$ and $\T_2$
  be finitely generated preorderings of $\A_1$ and $\A_2$, respectively. Suppose that the semialgebraic sets 
  \begin{align*}
    \K(\A_1,\T_1)
    &\coloneqq
    \set[\big]{ x=(x_1,\dots,x_n)\in \RR^n }{ f(x)\geq 0~~ \textup{for }~ f\in \T_1 },\\
    \K(\A_2,\T_2)
    &\coloneqq
    \set[\big]{ x'=(x_{n+1},\dots,x_{n+k})\in \RR^k }{ g(x')\geq 0~~ \textup{for }~ g\in \T_2 }
  \end{align*}
  are compact. Let\, $p\in \A_1\otimes \A_2\equiv \RR[x_1,\dots,x_{n+k}]$. If\, $p(x,x')>0$ for all 
  $(x,x')\in \K(\A_1,\T_1)\times\K(\A_2,\T_2)$, then $f\in \T_1\otimes \T_2$, that is, there exist 
  polynomials $p_1,\dots,p_r\in \T_1$ and $ q_1,\dots,q_r\in \T_2$, $r\in \NN$,
  such that 
  \begin{equation*}
    p(x,x')=\sum_{i=1}^r p_i(x)q_i(x')
    .
  \end{equation*}
\end{proposition}
\begin{proof}
  Obviously,\, $\S\coloneqq\T_1\otimes \T_2$\, is a semiring.
  Since the semialgebraic sets $\K(\A_1,\T_1)$ and $\K(\A_2,\T_2)$ are compact, the preorderings 
  $\T_1$ of $\A_1$ and $\T_2$ of $\A_2$ are Archimedean  (\cite{wm}, see  e.g.~\cite[Proposition 12.22]{sch17}).
  (This fact is the crucial step in the proof of the  Positivstellensatz in  \cite{sch91}.) Therefore, the elements
  $f\otimes \Unit$ and $\Unit\otimes g$ of $\A$, where $f\in \A_1$ and $g\in \A_2$, satisfy the Archimedean condition.
  Since these elements generate the algebra $\A=\A_1\otimes \A_2$,  Lemma \ref{archhilfs}
  implies that the semiring $\S$ is Archimedean. Clearly, $\characters(\A,\S)$ is the set of point evaluations at 
  $\K(\A_1,\T_1)\times\K(\A_2,\T_2)$. Now the assertion follows from Theorem \ref{archpos}, applied to the Archimedean semiring $\S$. 
\end{proof}

We close this section by associating a $C^*$-seminorm and a real $C^*$-algebra with an Archimedean semiring $\S$.
Recall that  we  defined\, $a \preceq b$\, if and only if\, $b-a\in \S$\, for $a,b\in \A$. Since $\S$ is Archimedean,
for any $a\in \A$ there is a $\lambda >0$ such that $-\lambda\Unit \preceq a \preceq \lambda\Unit$. Hence, for $a\in \A$,
we can define 
\begin{equation}\label{norminf}
  \norm{a}_\infty
  \coloneqq
  \inf\, \set[\big]{ \lambda \in (0,\infty) }{ {-\lambda}\Unit \preceq a \preceq \lambda\Unit }.
\end{equation}
Further, we set 
\begin{equation}\label{normdagger}
  \norm{a}_{\S^\dagger}
  \coloneqq
  \inf\, \set[\big]{ \lambda \in (0,\infty) }{ (\lambda^2\Unit - a^2)\in \S^\dagger } .
\end{equation}
\begin{lemma}\label{infcstarnorm}
~ $\norm{a}_\infty=\norm{a}_{\S^\dagger}$ for $a\in \A$.
\end{lemma}
\begin{proof}
  Since $\S^\dagger$ is a quadratic module (Theorem~\ref{archpre}),
  by \cite[Lemma~10.4]{sch20} we have $(\lambda^2-a^2)\in \S^\dagger$
  if and only if $(\lambda \Unit +a)\in \S^\dagger$ and $(\lambda \Unit- a)\in \S^\dagger$. Therefore, since 
  $\S\subseteq \S^\dagger$ and by comparing \eqref{norminf} and \eqref{normdagger}, it follows that $\norm{a}_{\S^\dagger} \leq \norm{a}_\infty$.

  To prove the converse inequality, let $\epsilon >0$. Then, by definition, 
  $\big((\norm{a}_{\S^\dagger} +\epsilon)^2\Unit - a^2 \big)\in \S^\dagger$ and therefore 
  $\big((\norm{a}_{\S^\dagger} +\epsilon)\Unit \pm\, a \big)\in \S^\dagger$ again by \cite[Lemma~10.4]{sch20},
  so that $\big((\norm{a}_{\S^\dagger} +\epsilon)\Unit \pm\, a +\epsilon \Unit \big)\in \S$ by the definition of $\S^\dagger$.
  Hence $\norm{a}_\infty \leq \norm{a}_{\S^\dagger} +2 \epsilon$ by \eqref{norminf}. Since this holds for all $\epsilon>0$,
  we obtain $\norm{a}_\infty \leq \norm{a}_{\S^\dagger}$. Putting both paragraphs together we get  $\norm{a}_\infty = \norm{a}_{\S^\dagger}$.
\end{proof}

\begin{proposition}
  For each Archimedean semiring, $\norm{a}_\infty$ is a $C^*$-seminorm on $\A$ and
  \begin{align}\label{realc}
    \norm{a}^2_\infty \leq \norm{a^2+b^2}_\infty~ \textup{ for all }~a,b\in \A.
  \end{align} 
\end{proposition}
\begin{proof}
  By Proposition \ref{archpre}, $\S^\dagger$ is an Archimedean quadratic module. Therefore, as shown by J.~Cimpric \cite{cimpric09}
  (see also \cite[Theorem 10.5]{sch20}), $\norm{\argument}_{\S^\dagger}$ is a  $C^*$-seminorm satisfying \eqref{realc}. Because of 
  Lemma~\ref{infcstarnorm} this holds also for  $\norm{\argument}_\infty$.
\end{proof}

Since $\norm{\argument}_\infty$ is a $C^*$-seminorm,  $\J\coloneqq \set{ a\in \A }{ \norm{a}_\infty = 0 }$ is an ideal and the quotient
algebra $\A/ \J$ carries the $C^*$-norm $\norm{\argument}_\S$ defined by ${\norm{a+\J}_\S} \coloneqq \norm{a}_\infty$, $a\in \A$. Because
\eqref{realc} holds, the completion of $(\A /\J,\norm{\argument}_\S)$ is a real $C^*$-algebra $\A_\S$. 
Thus, we have associated a {\it real $C^*$-algebra} $\A_\S$ with the Archimedean semiring $\S$. Hence the representation theory of
algebras with Archimedean quadratic modules developed in \cite{cimpric09} and  \cite[Chapter 10]{sch20} remains valid for
Archimedean semirings as well.

\section{Cofinal Elements and $c$-Localizable Semirings}\label{cofinal}

In this section we introduce two important concepts, cofinality and localizability of semirings.

\begin{definition}
  Let $\S$ be a semiring of $\A$. An element $s\in \Unit + \S$ is said to be \emph{cofinal in $\A$}
  if for all $a\in \A$ there exist numbers $\lambda \in (0,\infty)$ and $k\in \NN_0$ such that $\lambda s^k - a \in \S$.
\end{definition}
Clearly, if a semiring $\S$ of $\A$ contains a cofinal element $s\in \Unit + \S$, then $\S$ is generating.
Moreover, $\S$ is Archimedean if and only if the unit element $\Unit$ is cofinal in $\S$.

\begin{example}
  Let $\A=\RR[x_1,\dots,x_d]$ and let $\T$ be the preordering \eqref{preorderingf} generated by finitely many elements $f_1,\dots,f_r\in\A$, $r\in \NN$.
  Then $\T - \T = \A$ and, as noted in \cite[Corollary ~1.4]{M2}, an element $s \in \Unit + \T$ is cofinal in $\A$ if and only if there exist numbers 
  $\lambda \in (0,\infty)$ and $k\in \NN$ such that $\abs{t_i} \le \lambda s^k(t)$ for all $i\in \{1,\dots,d\}$ and all $t \in \RR^d$ satisfying 
  $f_j(t) \ge 0$ for all $j\in \{1,\dots,r\}$.
\end{example}

The next proposition shows that in the finitely generated case there always exists a cofinal element for any generating semiring.

\begin{proposition}\label{semiringgen}
  Let $\S$ be  a semiring of $\A$.
  Suppose the algebra $\A$ has finitely many generators $y_1,\dots, y_n$ and there exist elements $a_1,\dots, a_n \in \S$
  such that $a_j - y_j \in \S$ and $a_j + y_j \in \S$ for all $j\in \{1,\dots,n\}$. 
  Then $\S$ is generating and the element $s \coloneqq \Unit + \sum_{j=1}^n a_j$ is cofinal in $\A$.
\end{proposition}
\begin{proof}
  Obviously, $\S-\S$ is a subalgebra of $\A$. As 
  $y_j=\frac{1}{2}\big((a_j+y_j) - (a_j-y_j)\big)\in \S - \S$ holds for the generators $y_1,\dots,y_k$ of $\A$
  it follows that $\S-\S = \A$, so $\S$ is generating.
  
  We write $\preceq$ for the partial order $\preceq_\S$ from \eqref{preceq}.
  Let $a\in \A$. Then $a$ is a polynomial in the generators $y_j$, that is, $a=p(y_1,\dots,y_n)$ for some polynomial $p\in \RR[x_1,\dots,x_n]$.
  We express $p$ as a sum $p = \sum_{\ell = 0}^k p_\ell$ of homogeneous polynomials $p_\ell \in \RR[x_1,\dots,x_n]$ of degree $\ell$.
  By definition of $s$ we have $-s\preceq y_j\preceq s$ and $-s^0=-\Unit\preceq \Unit\preceq s^0=\Unit$. 
  Therefore,  Lemma~\ref{archhilfs} implies that $-s^\ell \preceq \prod_{j=1}^\ell y_{i_j} \preceq s^\ell$
  for all $\ell \in \{0,\dots,k\}$ and all $i_1, \dots, i_\ell \in \{1,\dots,n\}$. Hence
  there exist numbers $\lambda_0, \dots, \lambda_k \in (0,\infty)$ such that $-\lambda_\ell s^\ell \preceq p_\ell(y_1,\dots,y_n) \preceq \lambda_\ell s^\ell$
  holds for $\ell \in \{0,\dots,k\}$ (one may take for $\lambda_\ell$ the sum of absolute values of the coefficients of the polynomial $p_\ell$).
  Moreover, from $\Unit \preceq s$ it follows that $s^\ell \preceq s^{\ell+1}$ for all $\ell \in \NN_0$,
  and therefore $\sum_{\ell = 0}^k \lambda_\ell s^\ell \preceq \sigma s^k$ with $\sigma \coloneqq \sum_{\ell=0}^k \lambda_\ell$.
  Thus,  $-\sigma s^k \preceq p(y_1,\dots,y_n) \preceq \sigma s^k$. This proves that $s$ is cofinal in $\S$.
\end{proof}

In the case when $\S$ admits  a cofinal element, the cone $\S^\dagger$  can be nicely described by the following formula \eqref{sdd},
which generalizes formula \eqref{dagger1} from the Archimedean case.
\begin{proposition}\label{propsddagger}
  If $s\in \Unit + \S$ is a cofinal element of the  semiring $\S$, then
  \begin{equation}\label{sdd}
    \S^\dagger = \set[\big]{a\in \A}{\textup{there exists $n\in \NN$ such that }a+\epsilon s^n \in \S \textup{ for all }\epsilon \in {(0,\infty)}}
    .
  \end{equation}
\end{proposition}
\begin{proof}
  The inclusion ``$\supseteq$'' is clear. Conversely, let $a\in \S^\dagger$. Then there exists $b\in \S$ such that $a+\epsilon b \in \S$
  for all $\epsilon \in {(0,\infty)}$, and there exist numbers $\lambda \in (0,\infty)$ and $k\in \NN$ such that $\lambda s^k - b \in \S$.
  Consequently we have $a+\epsilon s^k = a+ (\epsilon/\lambda) b + (\epsilon/\lambda)( \lambda s^k - b ) \in \S$.
\end{proof}

\begin{definition}
  Let $\C$ be a convex cone of $\A$ and $c\in \C$. Then $\C$ is called \emph{$c$-localizable} if the following holds:
  Whenever $ca \in \C$ for some $a\in \A$, then $a\in \C$.
\end{definition}

\begin{proposition} \label{proposition:localization}
  Let $\S$ be a semiring of $\A$ and $s \in \S$, then
  \begin{equation*}
    \S_{s\mloc} 
    \coloneqq
    \set[\big]{a\in \A} {\textup{ there exists }k\in \NN_0\textup{ such that }s^k a \in \S}
  \end{equation*}
  is the smallest (with respect to inclusion $\subseteq$) $s$-localizable semiring of $\A$ that contains $\S$ as a subset.
\end{proposition}
\begin{proof}
  It is clear that $\S \subseteq \S_{s\mloc}$.
  To prove the converse inclusion, let $a,b\in \S_{s\mloc}$ and $\lambda, \mu \in {(0,\infty)}$.
  There are $k,\ell \in \NN_0$ such that $s^k a \in \S$ and $s^\ell b \in \S$.
  Then $s^{k+\ell}(\lambda a + \mu b) = \lambda s^{\ell} (s^k a) + \mu s^{k} (s^\ell b) \in \S$, and also $s^{k+\ell} a b = (s^k a)(s^\ell b) \in \S$,
  so $\S_{s\mloc}$ is a semiring of $\A$ that contains $\S$ as a subset. Moreover, $\S_{s\mloc}$ is $s$-localizable, because for any
  $a\in \A$ that fulfills $sa \in \S_{s\mloc}$ there exists $k\in \NN_0$ such that $s^{k+1} a = s^k(ca) \in \S$, and then also $a \in \S_{s\mloc}$.
  
  Conversely, if $\mathcal{T}$ is a $s$-localizable semiring of $\A$ satisfying $\S \subseteq \mathcal{T}$, then $\S_{s\mloc} \subseteq \mathcal{T}$,
  because for every $a \in \S_{s\mloc}$ there exists $k \in \NN_0$ such that $s^k a \in \S \subseteq \T$, and therefore $a\in \T$.
\end{proof}

An easy example of a $s$-localizable cone is the following.
\begin{example}
  Let $\A=\RR[x_1,\dots,x_n]$ and let $\T\coloneqq \set{ p\in \A }{ p(\xi)\geq 0 ~\textup{ for }~ \xi\in \RR^n}$.
  Then $\T$ is a preordering and $\T$ is $s$-localizable for every polynomial $s\in \T \setminus\{0\}$.
\end{example}
More interesting examples follow from Proposition \ref{proposition:localizableZ} below.

A convex cone $\C$ of a real vector space $V$ is called \emph{simplicial} if 
$\C$ is the convex cone generated by a finite linearly independent subset of $V$.
\begin{definition}
  A convex cone $\C$ of $\A$ is called \emph{filtered simplicial} if there exists a sequence $(\C_k)_{k\in \NN_0}$
  of simplicial convex cones of $\A$ such that $\C_k \subseteq \C_{k+1}$ for all $k\in \NN_0$ and $\bigcup_{k\in \NN_0} \C_k = \C$
  hold.
\end{definition}
An example of a filtered simplicial convex cone  is the convex cone of $\RR[x_1,\dots,x_n]$, $n\in \NN$,
of polynomials with non-negative coefficients (with the filtration by degree and the monomials as generators).

Filtered simplicial convex cones $\C$ of $\A$ are of interest for several reasons: Firstly, 
since $\C_k$ has a set of linearly independent generators, it is in principle  easy to decide whether an element belongs to $\C_k$ and so to $\C$,
and to compute $\inf \set{\lambda \in \RR}{a-\lambda\Unit \in \C_k}$ for $a,\Unit \in \C_k - \C_k$ by expressing $a$ and $\Unit$ in terms of
the basis of $\C_k - \C_k$ that is given by the linearly independent generators of $\C_k$. 
Secondly, they allow us to construct examples of localizable convex cones.

\begin{proposition} \label{proposition:localizableZ}
  Suppose $(\C_k)_{k\in \NN_0}$ is an increasing sequence of simplicial convex cones of $\A$ and
  $\C \coloneqq \bigcup_{k\in \NN_0} \C_k$ is the resulting filtered simplicial convex cone.
  Let $J$ be a non-empty subset of $\RR$.
  Let $\A[z] \cong \A \otimes \RR[z]$ denote the commutative unital $\RR$-algebra of polynomials in one variable $z$ with values in $\A$. From \eqref{eq:posJ} we recall
  \begin{equation}
    \label{eq:posj}
    \Pos(J) = \set[\big]{p \in \RR[z]}{p(t) \ge 0~\textup{ for all }~t\in J}
    .
  \end{equation}
  For every $k\in \NN_0$ let 
  \begin{equation*}
    \C_k \otimes \Pos(J)
    \coloneqq
    \set[\Big]{
      \sum\nolimits_{i=1}^\ell c_i p_i
    }{
      p_1, \dots, p_\ell \in \Pos(J)
    },
  \end{equation*}
  where $c_1, \dots,c_\ell \in \C_k$ are linearly independent generators of $\C_k$,
  and $\C \otimes \Pos(J) \coloneqq \bigcup_{k\in \NN_0} \C_k \otimes \Pos(J)$. Then, for any
  $q \in \RR[z]$ satisfying $q(t) > 0$ for all $t\in J$, the convex cone $\C \otimes \Pos(J)$ is $q$-localizable.
\end{proposition}
\begin{proof}
  The explicit formulas for $\C_k \otimes \Pos(J)$ and $\C \otimes \Pos(J)$ are easy to check.
  
  Let $q \in \RR[z]$ be given such that $q(t) > 0$ for all $t\in J$. Note that $q\neq 0$ because $J$ is non-empty.
  First consider $a \in \A[z]$ and assume that there are $p_1, \dots, p_\ell \in \RR[z]$ (not necessarily positive on $J$) 
  such that $q a = \sum_{i=1}^\ell c_i p_i$ with the generators $c_1, \dots,c_\ell \in \C_k$ for some $k\in \NN_0$.
  We will show that there are $r_1, \dots, r_\ell \in \RR[z]$ (not necessarily positive on $J$)
  such that $q r_i = p_i$ for $i\in \{1,\dots,\ell\}$ and $a = \sum_{i=1}^\ell c_i r_i$.
  As $q$ is a product of a scalar $\alpha \in \RR\setminus\{0\}$ and polynomials of the form $(z-\beta)^2+\gamma^2$ or $z-\delta$ with $\beta,\gamma,\delta \in \RR$,
  it suffices to discuss the two cases $q = (z-\beta)^2+\gamma^2$ and $q = z-\delta$:
  
  In the case $q = (z-\beta)^2+\gamma^2$ we consider the identity $q a = \sum_{i=1}^\ell c_i p_i$
  in the complexification $\A_\CC$ of the real algebra $\A$.
  Then we can set $z\coloneqq \beta \pm \I\gamma$ and obtain $0 = \sum_{i=1}^\ell c_i p_i(\beta\pm \I\gamma)$.
  Since the elements $c_1,\dots, c_\ell$ are also linearly independent in $\A_\CC$, we conclude that 
  $p_1(\beta\pm \I\gamma) = \dots = p_\ell(\beta\pm \I\gamma) = 0$ for both signs $\pm$. As $q = (z-(\beta + \I \gamma))(z-(\beta - \I \gamma))$
  there exist polynomials $r_i \in \RR[z]$ such that $q r_i = p_i$ for $i\in\{1,\dots,\ell\}$, so $q a = q \sum_{i=1}^\ell c_i r_i$.
  This means $a = \sum_{i=1}^\ell c_i r_i$ because $q \in \RR[z]\setminus\{0\} \subseteq \A[z]$ is not a zero-divisor of $\A[z]$.
  The case $q = z-\delta$ follows a slight modification of the preceding reasoning: Setting $z\coloneqq\delta$
  we obtain $p_i(\delta)=0$, so $p_i = (z-\delta) r_i$ for suitable $r_i \in \RR[z]$, $i\in \{1,\dots,\ell\}$.
  
  Now consider $a \in \A[z]$ such that $qa \in \C_k \otimes \Pos(J)$ for some $k\in \NN_0$. Then
  there are $p_1, \dots, p_\ell \in \RR[z]$ fulfilling $p_i(t) \ge 0$ for all $t\in J$
  such that $q a = \sum_{i=1}^\ell c_i p_i$. By the above considerations, there also are
  $r_1, \dots, r_\ell \in \RR[z]$ such that $q r_i = p_i$ for $i\in \{1,\dots,\ell\}$ and $a = \sum_{i=1}^\ell c_i r_i$.
  As $r_i(t) = p_i(t) / q(t) \ge 0$ for all $t\in J$, $i\in \{1,\dots,\ell\}$, it follows that $a \in \C_k \otimes \Pos(J)$.
\end{proof}
The following is an application of Proposition~\ref{proposition:localizableZ} that will be used later on.

\begin{corollary} \label{corollary:examplesimplex}
  Fix $\gamma \in {(0,\infty)}$ and let $\S$ be the semiring of $\RR[x_1,\dots,x_n]$ that is generated by
  the polynomials $\{x_1,\dots,x_n\} \cup \{\gamma-\sum_{j=1}^n x_j\}$. Let $J$ be a non-empty closed subset of $\RR$,
  and $q\in \RR[z]$ such that $q(t) > 0$ for all $t\in J$.
  Then $\S$ is a filtered simplicial and Archimedean semiring of $\RR[x_1,\dots,x_n]$
  and $\S \otimes \Pos(J)$ is a $q$-localizable semiring of
  $\RR[x_1,\dots,x_n,z]\cong \RR[x_1,\dots,x_n][z]$.
\end{corollary}
\begin{proof}
  Lemma~\ref{archhilfs} shows that $\S$ is Archimedean.
  For $k\in \NN_0$ let $\S_k$ be the convex cone of $\RR[x_1,\dots,x_n]$ that is generated 
  by all $g_{\ell_1,\dots,\ell_n} \coloneqq x_1^{\ell_1} \dots x_n^{\ell_n}(\gamma-\sum_{j=1}^n x_j)^{k-\ell_1-\dots - \ell_n}$
  with $\ell_1,\dots,\ell_n \in \NN_0$ such that $\ell_1+\dots +\ell_n \le k$. Then $\S_k \subseteq \S_{k+1}$ for all $k\in \NN_0$
  follows from multiplying any element of $\S_k$ with $1 = \gamma^{-1}x_1 + \dots + \gamma^{-1}x_n + \gamma^{-1}(\gamma-\sum_{j=1}^n x_j)$.
  Every monomial $x_1^{\ell_1} \dots x_n^{\ell_n}$ with $\ell_1,\dots,\ell_n \in \NN_0$ is an element of $\S_{\ell_1+\dots +\ell_n}$,
  hence $x_1^{\ell_1} \dots x_n^{\ell_n} \in \S_k$ whenever $\ell_1+\dots +\ell_n \le k$.
  It follows that $\S - \S$ is the linear subspace of $\RR[x_1,\dots,x_n]$
  of polynomials with degree at most $k$. Counting dimensions then shows that the generators $g_{\ell_1,\dots,\ell_n}$
  with $\ell_1,\dots,\ell_n \in \NN_0$, $\ell_1+\dots +\ell_n \le k$ are a basis of 
  this vector space, hence are linearly independent.
  This shows that $\S = \bigcup_{k\in \NN_0} \S_k$ is filtered simplicial.
  Consequently, $\S \otimes \Pos(J)$ is a $q$-localizable semiring of $\RR[x_1,\dots,x_n,z]$ by Proposition \ref{proposition:localizableZ}.
\end{proof}

By taking tensor products one can construct new examples of filtered simplicial convex cones.

\begin{proposition}
  Let $\A$ and $\B$ be commutative unital $\RR$-algebras, let $(\C_k)_{k\in \NN_0}$ and $(\mathcal{D}_k)_{k\in \NN_0}$
  be increasing sequences of simplicial convex cones of $\A$ and $\B$, respectively, and let
  $\C \coloneqq \bigcup_{k\in \NN_0} \C_k$ and $\mathcal{D} \coloneqq \bigcup_{k\in \NN_0} \mathcal{D}_k$ be the resulting filtered simplicial
  convex cones. Then $\C \otimes \mathcal{D} = \bigcup_{k\in \NN_0} \C_k \otimes \mathcal{D}_k$ is a filtered simplicial
  convex cone of $\A \otimes \B$.
\end{proposition}
\begin{proof}
  It is easily checked that $(\C_k \otimes \mathcal{D}_k)_{k\in \NN_0}$ is an increasing sequence of convex cones of
  $\A \otimes \B$ and that $\C \otimes \mathcal{D} = \bigcup_{k\in \NN_0} \C_k \otimes \mathcal{D}_k$. It only remains
  to show that $\C_k \otimes \mathcal{D}_k$ for arbitrary $k\in \NN_0$ is a simplicial convex cone:
  A suitable set of linearly independent generators of $\C_k \otimes \mathcal{D}_k$ is the set of tensor products
  of linearly independent generators of $\C_k$ and $\mathcal{D}_k$. This follows immediately from the observation that 
  $(\C_k \otimes \mathcal{D}_k) - (\C_k \otimes \mathcal{D}_k) = (\C_k - \C_k) \otimes (\mathcal{D}_k-\mathcal{D}_k)$ and by counting dimensions.
\end{proof}

\begin{corollary}
  The semiring of $\RR[x_1,\dots,x_n]$ which is generated by the polynomials $\bigcup_{i = 1}^n \{x_i,1-x_i\}$ is filtered simplicial.
\end{corollary}
\begin{proof}
  This semiring is the $k$-th tensor power of the semiring of $\RR[x]$  generated by $x$ and $1-x$, which is
  filtered simplicial by Corollary~\ref{corollary:examplesimplex}.
\end{proof}

\section{The filtered Positivstellensatz}\label{filtered}

First we develop the symbol algebra of a filtered commutative unital algebra.

\begin{definition}
  A \emph{filtered} algebra is a unital algebra $\A$ that is the union   $\A = \bigcup_{k\in \NN_0} \A^{(k)}$ of a sequence 
  of linear subspaces $\A^{(k)}$, $k\in \NN_0$, such that 
  \begin{align*}
  \Unit \in \A^{(0)}, ~~ \A^{(k)} \subseteq \A^{(k+1)},
   ~~ \A^{(k)} \A^{(\ell)} \subseteq \A^{(k+\ell)}\quad  \rm{for}\quad k,\ell \in \NN_0 .
  \end{align*}
  The sequence $(\A^{(k)})_{k \in \NN_0}$ is called the \emph{filtration} of the filtered algebra $\A$.

  For a filtered algebra $\A$ with filtration $(\A^{(k)})_{k \in \NN_0}$ we  define $\A^{(0)}_0 \coloneqq \{0\} \subseteq \A^{(0)}$ and
  \begin{align*}
    \A^{(k)}_0
    \coloneqq
    \set[\big]{a\in \A^{(k)}}{\textup{there exists $m\in \NN$ such that }a^m \in \A^{(km-1)}},~ k\in \NN.
  \end{align*}
  \end{definition}

\begin{lemma} \label{lemma:filteredObs}
  Suppose $\A$ is a filtered commutative unital $\RR$-algebra. Then  $\A^{(k)}_0$
  is a linear subspace of $\A^{(k)}$ for $k\in \NN_0$. We write $[\argument]^{k\msb} \colon \A^{(k)} \to \A^{(k)} / \A^{(k)}_0$
  for the canonical projection on the quotient vector space.
  Then, for all $k,\ell \in \NN_0$, we have $\A^{(k)}_0 \A^{(\ell)} \subseteq \A^{(k+\ell)}_0$ and
  the product of $\A$ gives well-defined  bilinear maps $\A^{(k)} / \A^{(k)}_0 \times \A^{(\ell)} / \A^{(\ell)}_0 \to \A^{(k+\ell)} / \A^{(k+\ell)}_0$,
  \begin{equation*}
    \big( [a]^{k\msb} ,  [b]^{\ell\msb} \big) \mapsto  [a]^{k\msb}  [b]^{\ell\msb} \coloneqq [ab]^{(k+\ell)\msb}.
  \end{equation*}
\end{lemma}
\begin{proof}
  Given $k\in \NN$ and $a,b \in \A^{(k)}_0$,  there exist $m,n\in \NN$ such that $a^m \in \A^{(km-1)}$ and $b^n \in \A^{(kn-1)}$. 
  By the binomial theorem,  $(a+b)^{m+n}$ is a linear combination of elements $a^{m'} b^{n'}$, where $m',n' \in \NN_0$
  and $m'+n' = m+n$.  This implies that $(a+b)^{m+n} \in \A^{(k(m+n)-1)}$, so  $a+b \in \A^{(k)}_0$. It is obvious that 
  $\lambda a \in \A^{(k)}_0$ for $\lambda \in \RR$. This shows that $\A^{(k)}_0$ is a linear subspace of $\A^{(k)}$.
  
  Next let $a \in \A^{(k)}_0$ and $b\in \A^{(\ell)}$ for $k\in \NN$ and $\ell \in \NN_0$. Then
  there exists $m\in \NN$ such that $a^m \in \A^{(km-1)}$ and therefore $(ab)^m = a^m b^m \in \A^{(km-1)} \A^{(\ell m)} \subseteq \A^{((k+\ell)m-1)}$,
  so $ab \in \A^{(k+\ell)}_0$. It follows that $\A^{(k)}_0 \A^{(\ell)} \subseteq \A^{(k+\ell)}_0$. In the  case $k=0$
  it is clear that $\A^{(0)}_0 \A^{(\ell)} = \{0\} \subseteq \A^{(\ell)}_0$ for  $\ell \in \NN_0$.
  
  The preceding implies that the quotient vector spaces $\A^{(k)} / \A^{(k)}_0$ exist and that
  the bilinear maps $\A^{(k)} \times \A^{(\ell)} \ni (a,b) \mapsto ab \in \A^{(k+\ell)}$ descend to the quotients for all $k, \ell \in \NN_0$.
\end{proof}

These bilinear maps are now put together to construct a new algebra  of ``symbols'' which encodes information about the behaviour of algebra elements ``at $\infty$''.

\begin{definition}
  Let $\A$ be a filtered commutative unital $\RR$-algebra with filtration $(\A^{(k)})_{k \in \NN_0}$.
  We equip the graded vector space $\A^\sb \coloneqq \bigoplus_{k\in \NN_0} \A^{(k)} / \A^{(k)}_0$
  with the product $\A^\sb \times \A^\sb \to \A^\sb$,
  \begin{equation}
    \Bigg( \bigoplus_{k\in \NN_0} [a_k]^{k\msb},  \bigoplus_{\ell\in\NN_0} [b_\ell]^{\ell\msb} \Bigg)
    \mapsto
    \Bigg( \bigoplus_{k\in \NN_0}  [a_k]^{k\msb} \Bigg) \Bigg( \bigoplus_{\ell\in\NN_0} [b_\ell]^{\ell\msb} \Bigg)
    \coloneqq
    \bigoplus_{m\in\NN_0} \Big[ \sum\nolimits_{n=0}^m a_{m-n} b_n \Big]^{m\msb},
  \end{equation}
  and call $\A^\sb$ the \emph{algebra of symbols} of $\A$.
\end{definition}
  
We identify each quotient vector space $\A^{(k)}/\A^{(k)}_0$  with the corresponding linear subspace of $\A^\sb$.
Then the canonical projections  $[\argument]^{k\msb}$ from Lemma~\ref{lemma:filteredObs}
become linear maps $[\argument]^{k\msb} \colon \A^{(k)} \to \A^\sb$, which will be called  \emph{symbol maps}.

It is easily checked that $\A^\sb$ is indeed a commutative unital $\RR$-algebra with unit $\Unit_{\A^\sb} = [\Unit_\A]^{0\msb}$.
Under the identification of the quotient vector spaces $\A^{(k)}/\A^{(k)}_0$ with subspaces of $\A^\sb$,
the multiplication of $\A^\sb$ can  be rewritten as
\begin{equation*}
  \Big( \sum\nolimits_{k \in \NN_0} [a_k]^{k\msb} \Big) \Big( \sum\nolimits_{\ell\in \NN_0} [b_\ell]^{\ell\msb} \Big)
  =
  \sum_{m \in \NN_0} \Big[ \sum\nolimits_{n=0}^m a_{m-n} b_n\Big]^{m\msb}
   =
  \sum_{k,\ell \in \NN_0} \big[a_k b_\ell\big]^{(k+\ell)\msb}
\end{equation*}
for finite sequences $(a_k)_{k\in \NN_0}$, $(b_\ell)_{\ell\in \NN_0}$, where $a_k \in \A^{(k)}$, $b_\ell \in \A^{(\ell)}$. 


\begin{example} \label{example:filtrationsbydegree}
  Let $\A\coloneqq\RR[x_1,\dots,x_n]$, $n\in \NN$. We assign a number $\gamma_j \in [0,\infty)$, called the \emph{degree},
  to each generator $x_j$, $j\in \{1,\dots,n\}$. Using the standard multi-index notation $x^\ell=x_1^{\ell_1,}\cdots x_n^{\ell_n}$
  for $\ell \in \NN_0^n$ we define 
  \begin{equation*}
    \RR[x_1,\dots,x_n]^{(k)}
    \coloneqq
    {\Lin \set[\Big]{ x^\ell }{ \ell \in \NN_0^n \textup{ such that } \sum\nolimits_{j=1}^n \gamma_j \ell_j \le k } }
  \end{equation*}
  for $k\in \NN_0$. Then $\RR[x_1,\dots,x_n]$ becomes a filtered  $\RR$-algebra with filtration $(\RR[x_1,\dots,x_n]^{(k)})_{k\in \NN_0}$.
  In this case it is easy to check that
  \begin{equation*}
    \RR[x_1,\dots,x_n]^{(k)}_0
    =
    \Lin \set[\Big]{ x^\ell }{ \ell \in \NN_0^n \textup{ such that } \sum\nolimits_{j=1}^n \gamma_j \ell_j < k  }
  \end{equation*}
  for $k\in \NN$. The quotients $\RR[x_1,\dots,x_n]^{(k)} / \RR[x_1,\dots,x_n]^{(k)}_0$  can then be identified with the complementary linear subspaces
  \begin{equation*}
    \RR[x_1,\dots,x_n]^k
    \coloneqq
    \Lin  \set[\Big]{ x^\ell }{ \ell \in \NN_0^n \textup{ such that } \sum\nolimits_{j=1}^n \gamma_j \ell_j = k }
  \end{equation*}
  of $\RR[x_1,\dots,x_n]^{(k)}_0$ in $\RR[x_1,\dots,x_n]^{(k)}$ such that $\RR[x_1,\dots,x_n]^{(k)} = \RR[x_1,\dots,x_n]^k \oplus \RR[x_1,\dots,x_n]^{(k)}_0$
  and the symbol maps then become the projections $[\argument]^{k\msb} \colon \RR[x_1,\dots,x_n]^{(k)} \to \RR[x_1,\dots,x_n]^k$
  with respect to this decomposition. Thus, the symbol algebra is the unital subalgebra $\RR[x_1,\dots,x_n]^\sb = \sum_{k\in \NN_0} \RR[x_1,\dots,x_n]^k$.
\end{example}

Next we extend the construction of the symbol algebra to modules of semirings of the algebra.

\begin{definition} \label{definition:Csb}
  Let $\A$ be a filtered commutative unital $\RR$-algebra and $\C$ a convex cone of $\A$ with $\Unit \in \C$. Then we define
  \begin{equation}
    \C^\sb \coloneqq \set[\Big]{\sum\nolimits_{k=0}^K [c_k]^{k\msb}}{K\in\NN_0;\,c_k \in \C \cap \A^{(k)}\textup{ for }k\in \{0,\dots,K\} }
    .
  \end{equation}
\end{definition}

\begin{proposition} \label{proposition:Csb}
  Let $\A$ be a filtered commutative unital $\RR$-algebra, $\S$ a semiring of $\A$ and $\C\subseteq A$ an $\S$-module.
  Then $\S^\sb$ is a semiring of $\A^\sb$ and $\C^\sb$ is an $\S^\sb$-module of $\A^\sb$.
\end{proposition}
\begin{proof}
  First consider a convex cone $\C$ of $\A$. Then it is easy to check that $\C^\sb$ is a convex cone of $\A^\sb$,
  and if $\Unit_\A \in \C$, then $\Unit_{\A^\sb} = [\Unit_\A]^{0\msb} \in \C^\sb$.
  
  Now let $\S$ be a semiring of $\A$  and assume that $\C\subseteq \A$ is an $\S$-module.
  Then, for $K,L \in \NN_0$ and $c_k\in \C \cap \A^{(k)}$, $s_\ell \in \S \cap \A^{(\ell)}$ for all $k\in \{0,\dots,K\}$, $\ell\in \{0,\dots,L\}$ we  obtain
  \begin{equation*}
    \bigg( \sum_{k=0}^K [c_k]^{k\msb} \bigg)\bigg( \sum_{\ell=0}^L [s_\ell]^{\ell\msb} \bigg) = \sum_{m=0}^{K+L} \sum_{n=0}^m [c_{m-n}s_m]^{m\msb} \in \C^\sb
  \end{equation*}
  by the definition of the multiplication of $\A^\sb$ and of $\C^\sb$. For $\C = \S$ this shows that $\S^\sb$ is a semiring of $\A^\sb$, and in the general case that
  $\C^\sb$  is an $\S^\sb$-module.
\end{proof}

In order  to state the filtered Positivstellensatz, we  need one more definition.

\begin{definition}
  Let $\A$ be a filtered commutative unital $\RR$-algebra and $\S$ a generating semiring of $\A$. An element $s \in (\Unit + \S) \cap \A^{(1)}$
  is called \emph{compatible with the filtration of $\A$} if for all $k\in \NN_0$ and  $a\in\A^{(k)}$ there exists a number $\lambda\in(0,\infty)$ such 
  that $\lambda s^k - a \in \S$.
\end{definition}
Note that in this case the element $s$ is in particular  cofinal in $\A$.
This compatibility condition can be considered as a generalization of the Archimedean property: 
A semiring $\S$ of a commutative unital $\RR$-algebra $\A$ is Archimedean if and only if $\Unit$
is compatible with the trivial filtration on $\A$ which is defined by $\A^{(k)} \coloneqq \A$ for all $k\in\NN_0$.

\begin{example}
  Let $\A$ be a commutative unital $\RR$-algebra and $\S$ a semiring of $\A$. Suppose that $s \in \Unit + \S$ is cofinal in $\S$.
  Then, using Lemma~\ref{archhilfs} one can easily check that for all $k\in \NN_0$,
  \begin{equation*}
    \A^{(k)}
    \coloneqq
    \set[\big]{a\in\A}{\textup{there exists }~\lambda\in(0,\infty)~\textup{ such that }~{-\lambda s^k} \preceq a \preceq \lambda s^k }
  \end{equation*}
  are linear subspaces of $\A$, turning $\A$ into a filtered commutative unital $\RR$-algebra.
  For this construction it is clear that $s \in (\Unit + \S) \cap \A^{(1)}$ and that $s$ is compatible with the filtration of $\A$.
\end{example}
However, for obtaining explicit examples it is often easier to fix the filtration first and then find a (necessarily cofinal)
element compatible with the filtration. This will be discussed in Section~\ref{examples}.

The main result of this section is the following filtered Positivstellensatz.

\begin{theorem} \label{theorem:filteredPositivstellsatz}
  Let $\A$ be a filtered commutative unital $\RR$-algebra, $\S$ a generating semiring of $\A$ and $\C\subseteq \A$ an $\S$-module.
  Suppose $s\in (\Unit+\S) \cap \A^{(1)}$ is compatible with the filtration of $\A$. Then, for any $a\in\A^{(k)}$, $k\in \NN_0$, the following are equivalent:
  \begin{enumerate}
    \item \label{item:filteredPositivstellsatz:ana}
      $\varphi(a) > 0$ for all $\varphi \in \characters(\A,\C)$ and $\psi([a]^{k\msb}) > 0$ for all $\psi \in \characters(\A^\sb,\C^\sb)$ with $\psi([s]^{1\msb}) = 1$.
    \item \label{item:filteredPositivstellsatz:alg}
      There exist numbers $\epsilon \in (0,\infty)$ and $\ell \in \NN_0$ such that $s^\ell a \in \epsilon s^{k+\ell} + \C$.
  \end{enumerate}
\end{theorem}
In order to prove this theorem, we  construct an auxiliary algebra $\A[z]^\bd$ endowed with a module of an Archimedean semiring.
Like in Proposition~\ref{proposition:localizableZ} we write $\A[z] \cong \A \otimes \RR[z]$ for the commutative unital $\RR$-algebra
of polynomials in one variable $z$ with coefficients in $\A$.

\begin{definition}
  Let $\A$ be a filtered commutative unital $\RR$-algebra. We define $\A[z]^\bd$ to be the unital subalgebra of
  $\A[z]$ that is generated by $\bigcup_{k\in \NN_0} \set{a z^k}{a \in \A^{(k)}}$.
\end{definition}
More explicitly, one finds that
\begin{equation*}
  \A[z]^\bd
  =
  \set[\Big]{
    \sum\nolimits_{k=0}^K a_k z^k
  }{
    K\in \NN_0;\,a_k \in \A^{(k)} \textup{ for }k\in \{0,\dots,K\}
  }
  .
\end{equation*}
The crucial property of this auxiliary algebra is that each of its characters can be lifted  either to the original algebra $\A$ or to the algebra of symbols $\A^\sb$:

\begin{lemma} \label{lemma:liftingcharacters}
  Let $\A$ be a filtered commutative unital $\RR$-algebra and $\rho \in \characters(\A[z]^\bd)$.
  \begin{enumerate}
    \item If $\rho(z) \neq 0$, then there exists a unique $\hat\rho \in \characters(\A)$ such that
      $\rho(a z^k) = \hat\rho(a) \,\rho(z)^k$  for all $k\in \NN_0$ and  $a\in \A^{(k)}$.
    \item If $\rho(z) = 0$, then there exists a unique $\hat\rho \in \characters(\A^\sb)$ such that
      $\rho(a z^k) = \hat\rho([a]^{k\msb})$  for all $k\in \NN_0$ and $a\in \A^{(k)}$.
  \end{enumerate}
\end{lemma}
\begin{proof}
  First we consider the case $\rho(z) \neq 0$. Let $a\in \A$ and $k,k' \in \NN_0$ be given such that $a \in \A^{(k)}$ and $k\le k'$, hence also $a \in \A^{(k')}$. Then
   $\rho(a z^{k'}) \, \rho(z)^{-k'} = \rho(a z^{k})\,\rho(z^{k'-k}) \, \rho(z)^{-k'} = \rho(a z^k) \, \rho(z)^{-k}$.
  It follows that the map $\hat\rho \colon \A \to \RR$,
  \begin{equation*}
    a \mapsto \hat\rho(a) \coloneqq \rho(az^k) \, \rho(z)^{-k},
    \quad\textup{for  any $k\in\NN_0$ such that $a\in\A^{(k)},$}
  \end{equation*}
  is well-defined. It is easy to check that $\hat\rho$ is indeed a unital algebra homomorphism and that
  $\rho(a z^k) = \hat\rho(a) \,\rho(z)^k$ holds for all $k\in \NN_0$ and  $a\in \A^{(k)}$ by construction.
  It is also clear that $\hat\rho$ is uniquely determined by this identity.
  
  Now we treat the case $\rho(z) = 0$. Let $k\in \NN$ and $a \in \A^{(k)}_0$. Then there exists $m\in \NN_0$
  such that $a^m \in \A^{(km-1)}$ and therefore $\rho(a z^k)^m = \rho(z)\, \rho(a^m z^{km-1}) = 0$, hence $\rho(a z^k) = 0$.
  In the special case $k=0$, the identity $\rho(a z^0)=0$ holds trivially for $a \in \A^{(0)}_0 = \{0\}$.
  Therefore, the map $\hat\rho \colon \A^\sb \to \RR$,
  \begin{equation*}
    \sum_{k = 0}^\infty [a_k]^{k\msb}
    \mapsto
    \hat\rho\Big( \sum\nolimits_{k=0}^\infty [a_k]^{k\msb} \Big)
    \coloneqq
    \sum_{k=0}^\infty \rho(a_k z^k),
  \end{equation*}
  is well-defined. Again, it is easily checked that $\hat\rho$ is a unital algebra homomorphism and
  $\rho(a z^k) = \hat\rho([a]^{k\msb})$ for all $k\in \NN_0$ and  $a\in \A^{(k)}$.
  Clearly, $\hat\rho$ is uniquely determined by this property.
\end{proof}
On this auxiliary algebra we  construct a module of an Archimedean semiring as follows.

\begin{definition}\label{defczbds}
  Let $\A$ be a filtered commutative unital $\RR$-algebra, $\S$ a generating semiring of $\A$ and $\C\subseteq \A$ an $\S$-module.
  Suppose $s\in (\Unit+\S) \cap \A^{(1)}$ is compatible with the filtration of $\A$. We define
  \begin{align*} 
    \C[z]^\bd_s
    \coloneqq
    \set[\Big]{
      \sum\nolimits_{k=0}^K
      a_k z^k
    }{
      K\in \NN_0;\,a_k \in \A^{(k)}\textup{ for all $k\in\{0,\dots,K\}$ such that }\sum\nolimits_{k=0}^K s^{K-k} a_k \in \C
    }.
  \end{align*}
\end{definition}

Clearly,  $\C[z]^\bd_s$ is a subset of $\A[z]^\bd$.
Note that this definition does not only apply to the $\S$-module $\C$, but also to the semiring $\S$ itself.

\begin{lemma} \label{lemma:filteredPositivstellsatz}
  Let $\A$ be a filtered commutative unital $\RR$-algebra, $\S$ a generating semiring of $\A$ and $\C$ an $\S$-module of $\A$.
  Moreover, let $s\in (\Unit+\S) \cap \A^{(1)}$ be compatible with the filtration of $\A$.
  Then $\S[z]^\bd_s$ is an Archimedean semiring of $\A[z]^\bd$ and $\C[z]^\bd_s$ is an $\S[z]^\bd_s$-module of $\A[z]^\bd$.
\end{lemma}
\begin{proof}
  First we consider a convex cone $\C$ of $\A$ and construct $\C[z]^\bd_s$ as in Definition \ref{defczbds}.
  Suppose $K\in \NN_0$ and $a_k \in \A^{(k)}$ for $k\in\{0,\dots,K\}$ such that $\sum_{k=0}^K s^{K-k} a_k \in \C$. Then
   also $\sum_{k=0}^{K'} s^{K'-k} a_k \in \C$ for all $K'\in \NN_0$ with $K'\ge K$ if one sets $a_k \coloneqq 0$
  for all $k\in \NN$, $k> K$. It is now easy to show that $\C[z]^\bd_s$ is closed under addition and  under multiplication with nonnegative scalars. Thus, $\C[z]^\bd_s$ is a convex cone of $\A[z]^\bd$.
  As $\Unit_\A \in \C$, also $\Unit_{\A[z]^\bd} = \Unit_\A z^0 \in \C[z]^\bd_s$.
  
  Now let $\S$ be a semiring of $\A$ and assume that $\C\subseteq \A$ is an $\S$-module. By the preceding, $\C[z]^\bd_s$ and $\S[z]^\bd_s$ are convex cones of $\A[z]^\bd_s$
  and contain $\Unit_{\A[z]^\bd}$.
  Let  $L\in \NN_0$ and $b_\ell \in \A^{(\ell)}$ for all $\ell\in\{0,\dots,L\}$ be given such that $\sum_{\ell=0}^L s^{L-\ell} b_\ell \in \S$.
  We set $a_k \coloneqq 0$ for all $k \in \NN_0$ with $k>K$ and $b_\ell \coloneqq 0$ for all $\ell\in \NN_0$ with $\ell> L$.
  Then we find that $(\sum_{k=0}^K a_k z^k)(\sum_{\ell=0}^L b_\ell z^\ell) = \sum_{m=0}^{K+L} \sum_{n=0}^m a_{m-n} b_{n} z^m \in \C[z]^\bd_s$
  because $\sum_{m=0}^{K+L} s^{K+L-m}\sum_{n=0}^m a_{m-n} b_{n} = (\sum_{k=0}^K s^{K-k} a_k)(\sum_{\ell=0}^L s^{L-\ell} b_\ell )\in \C$.
  For $\C = \S$ this shows that $\S[z]^\bd_s$ is a semiring of $\A[z]^\bd_s$ and  in the general case that $\C[z]^\bd_s$ is an $\S[z]^\bd_s$-module of $\A[z]^\bd_s$.
  
  Now consider an arbitrary element $\sum\nolimits_{k=0}^K a_k z^k$ of $\A[z]^\bd$, i.e.~$K\in \NN_0$ and $a_k \in \A^{(k)}$ for  $k\in\{0,\dots,K\}$.
  By the compatibility of $s$ with the filtration there exist numbers $\lambda_0,\dots,\lambda_K\in(0,\infty)$
  such that $\lambda_k s^k - a_k \in \S$ for all $k\in \{0,\dots, K\}$. Therefore, $\lambda_k \Unit - a_k z^k \in \S[z]^\bd_s$ for  $k\in \{0,\dots, K\}$.
  Setting $\lambda_{\Sigma} \coloneqq \sum_{k=0}^K \lambda_k$, we have $\lambda_\Sigma \Unit - \sum\nolimits_{k=0}^K a_k z^k \in \S[z]^\bd_s$.
  This shows that $\S[z]^\bd_s$ is Archimedean.
\end{proof}
We are now ready to give the proof of the filtered Positivstellensatz.

\begin{proof}[of Theorem~\ref{theorem:filteredPositivstellsatz}]
  
\refitem{item:filteredPositivstellsatz:alg}$\implies$\refitem{item:filteredPositivstellsatz:ana}:  Assume that there exist $\epsilon \in (0,\infty)$ and $\ell \in \NN_0$ such that $s^\ell a \in \epsilon s^{k+\ell} + \C$.
  Let $\varphi \in \characters(\A,\C)$. Then $\varphi(s)^\ell \,\varphi(a) \ge \epsilon \varphi(s)^{k+\ell}$, so
  $\varphi(a) \ge \epsilon \varphi(s)^k \ge \epsilon > 0$ because $\varphi(s) \ge 1$. Moreover, from $s^\ell a- \epsilon s^{k+\ell} \in \C \cap \A^{(k+\ell)}$
  it follows that $[s^\ell a - \epsilon s^{k+\ell}]^{(k+\ell)\msb} \in \C^\sb$, so 
  $\psi([s]^{1\msb})^\ell \, \psi([a]^{k\msb}) \ge \epsilon \psi([s]^{1\msb} )^{k+\ell} = \epsilon > 0$ for all $\psi \in \characters(\A^\sb,\C^\sb)$ satisfying $\psi([s]^{1\msb} ) = 1$.
  
\refitem{item:filteredPositivstellsatz:ana}$\implies$\refitem{item:filteredPositivstellsatz:alg}: Assume $\varphi(a) > 0$ for all $\varphi \in \characters(\A,\C)$ and $\psi([a]^{k\msb}) > 0$ for all $\psi \in \characters(\A^\sb,\C^\sb)$ such that $\psi([s]^{1\msb} ) = 1$.
  From $a \in \A^{(k)}$ it follows that $a z^k \in \A[z]^\bd$. Now let $\rho \in \characters(\A[z]^\bd, \C[z]^\bd_s)$ be given and consider
  its lift $\hat\rho$ from Lemma~\ref{lemma:liftingcharacters}. We distinguish two different cases:
  
  If $\rho(z) \neq 0$, then $\hat\rho \in \characters(\A)$. Note that $z\in \C[z]^\bd_s$ because $\Unit \in \C$, so $\rho(z)>0$.
  For all $c\in \C$ there exists $\ell\in \NN_0$ such that $c \in \A^{(\ell)}$. Then  $c z^\ell \in \C[z]^\bd_s$, so $\rho(c z^\ell) \ge 0$.
  It follows that $\hat\rho(c) = \rho(c z^\ell)\,\rho(z)^{-\ell} \ge 0$ for all $c\in \C$, i.e.~$\hat\rho\in\characters(\A,\C)$. Therefore, $\rho(a z^\ell) = \hat\rho(a) \, \rho(z)^\ell > 0$ because $\hat\rho(a) > 0$ by assumption.
  
  If $\rho(z) = 0$, then $\hat\rho \in \characters(\A^\sb)$.
  For $\ell\in \NN_0$ and $c \in \C \cap \A^{(\ell)}$ it follows from $c z^k \in \C[z]^\sb_s$ that $\hat\rho([c]^{k\msb}) = \rho( c z^k ) \ge 0$.
  This shows that $\hat\rho \in \characters(\A^\sb,\C^\sb)$. Moreover, from $sz- \Unit_{\A[z]^\bd} \in \C[z]^\bd_s$ and $\Unit_{\A[z]^\bd}- sz \in \C[z]^\bd_s$
  we obtain $\hat\rho([s]^{1\msb}) = \rho(sz) = \rho(\Unit_{\A[z]^\bd}) = 1$, therefore $\rho(a z^k) = \hat\rho([a]^{k\msb}) > 0$ by assumption.
  
  By the Positivstellensatz for modules of Archimedean semirings (Theorem \ref{archpos}), applied to the $\S[z]^\bd_s$-module $\C[z]^\bd_s$
  (see Lemma \ref{lemma:filteredPositivstellsatz}), there exists $\epsilon \in (0,\infty)$
  such that $a z^k \in \epsilon\Unit_{\A[z]^\bd} + \C[z]^\bd_s$, i.e.~there is a number $K\in \NN_0$, $K\ge k$, such that
  $s^{K-k} a - \epsilon s^K \in \C$. Setting $\ell \coloneqq K-k \in \NN_0$ yields $s^\ell a \in \epsilon s^{k+\ell} + \C$.
\end{proof}
The filtered Positivstellensatz  yields a description of pointwise non-negative elements similar to \cite[Theorem 2.2]{M2}.

\begin{corollary} \label{corollary:marshall}
  Let $\A$ be a filtered commutative unital $\RR$-algebra, $\S$ a generating semiring of $\A$ and $\C$ an $\S$-module of $\A$.
  Moreover, let $s\in \Unit+\S$ be cofinal in $\S$ and let $\C_{s\mloc}$ be the $s$-localization of $\C$
  from Proposition~\ref{proposition:localization}. Then 
  \begin{equation*}
    \big( \C_{s\mloc} \big)^\dagger = \set[\big]{a\in \A}{\varphi(a) \ge 0~~\textup{for all}~~\varphi\in\characters(\A,\C)}.
  \end{equation*}
\end{corollary}
\begin{proof}
  The inclusion ``$\subseteq$'' is clear. Conversely, let $a\in \A$ and suppose that $\varphi(a) \ge 0$ for all $\varphi \in \characters(\A,\C)$.
  There exists $k\in \NN_0$ such that $a\in \A^{(k)}$. Let $\epsilon \in (0,\infty)$. We show that $a+\epsilon s^{k+1} \in \C_{s\mloc}$.
  Indeed, $\varphi(a+\epsilon s^{k+1}) \ge \epsilon \varphi(s)^{k+1} \ge \epsilon > 0$ holds for  $\varphi \in \characters(\A,\C)$,
  and $[a+\epsilon s^{k+1}]^{(k+1)\msb} = \epsilon [s^{k+1}]^{(k+1)\msb}$, so that
  $\psi( [a+\epsilon s^{k+1}]^{(k+1)\msb} ) = \epsilon \psi([s]^{1\msb})^{k+1} = \epsilon$ for all $\psi \in \characters(\A^\sb,\C^\sb)$ satisfying $\psi([s]^{1\msb}) = 1$.
  Then, by Theorem~\ref{theorem:filteredPositivstellsatz}, there exists $\ell \in \NN_0$ such that $s^\ell (a + \epsilon s^{k+1}) \in \epsilon s^{k+1+\ell} + \C \subseteq \C$,
  so $a + \epsilon s^{k+1} \in \C_{s\mloc}$. Thus, $a\in (\C_{s\mloc})^\dagger$.
\end{proof}

\section{Examples for the filtered Positivstellensatz}\label{examples}

The following simple lemma helps to check the compatibility condition of cofinal elements.

\begin{lemma} \label{lemma:compatibility}
  Let $\A$ be a filtered commutative unital $\RR$-algebra, $\S$ a generating semiring of $\A$, and $s \in (\Unit + \S) \cap \A^{(1)}$.
  Moreover, assume that for every $k\in \NN$, the space $\A^{(k)}$ is the linear span of the set
  \begin{equation*}
    \big(\A^{(1)}\big)^k \coloneqq \set[\big]{ a_1 \cdots a_k }{ a_1,\dots,a_k \in \A^{(1)}}
    .
  \end{equation*}
  Then $s$ is compatible with the filtration of $\A$ if and only if for each $a\in\A^{(1)}$ there exists a number $\lambda\in(0,\infty)$ such 
  that $\lambda s - a \in \S$.
\end{lemma}
\begin{proof}
  The ``only if''-part is clear from the  compatibility definition with the filtration, while the ``if''-part is an easy
  application of Lemma~\ref{archhilfs}.
\end{proof}

In Theorem~\ref{theorem:filteredPositivstellsatz}, one might be
tempted to replace condition \refitem{item:filteredPositivstellsatz:ana} by
\begin{enumerate}
  \item[\textit{(i')}] \label{item:filteredPositivstellsatz:anaAlt}
 \textit{   There exists $\epsilon\in (0,\infty)$ such that $\varphi(a) > \epsilon \varphi(s)^k$ for all $\varphi \in \characters(\A,\C)$.}
\end{enumerate}
The following rather pathological example shows that this is not true in general. Nevertheless it would be interesting
to find general sufficient conditions which imply that \textit{(i)} can  be replaced by \textit{(i')}.

\begin{example} \label{example:pathological}
  Let $\A = \RR[x]$ be the polynomial algebra in one variable  with the usual filtration by degree,
  i.e.~$\RR[x]^{(k)}$ is the linear span of $\{x^0, \dots, x^k\}$. It is easily checked that
  \begin{equation*}
    \S \coloneqq \{0\} \cup \set[\Big]{ \sum\nolimits_{k=0}^K a_k x^k }{ K \in \NN_0;\,a_0, \dots,a_K \in \RR \textup{ and } a_K > 0 }
  \end{equation*}
  is a generating semiring of $\A$. Then,  $\characters(\A,\S) = \emptyset$. Indeed, 
  we have $\varphi\big( x - (\varphi(x)+1)\Unit \big) = -1$ for any $\varphi \in \characters(\A)$.
  Since $x - (\varphi(x)+1)\Unit \in \S$, it follows that $\varphi\notin\characters(\A,\S)$.
  
  Then $s \coloneqq \Unit + x \in (\Unit+\S) \cap \A^{(1)}$ is cofinal and compatible with the filtration of $\A$
  by Lemma~\ref{lemma:compatibility}.
  The corresponding symbol algebra $\RR[x]^\sb$ is isomorphic to $\RR[x]$ and the symbol map
  $[\argument]^{k\msb} \colon \RR[x]^{(k)} \to \RR[x]^\sb$
  is the projection onto the space of $k$-homogeneous polynomials, see Example~\ref{example:filtrationsbydegree}, and 
  \begin{equation*}
  \S^\sb = \{0\} \cup \set[\Big]{ \sum\nolimits_{k=0}^K a_k x^k }{ K \in \NN_0;\,a_0, \dots,a_K \in (0,\infty) }.
  \end{equation*}
  Let $\chi_1$ be the evaluation functional at $x=1$. Then $\chi_1\in\characters(\A^\sb,\S^\sb)$ and  $\chi_1([s]^{1\msb}) = \chi_1(x) = 1$.
  
  Clearly, $-s^\ell \notin \epsilon s^\ell + \S$ for all $\ell \in \NN_0$.
  For $k\coloneqq 0$, $a \coloneqq -\Unit \in \A^{(0)}$, $\epsilon \coloneqq 1$, we have $s^\ell a = -s^\ell \notin \epsilon s^{k+\ell} + \S$ for all $\ell \in \NN_0$.
  Since $ \characters(\A,\S)=\emptyset$,  $\varphi(a) \ge \epsilon \varphi( s )^k$ for all $\varphi \in \characters(\A,\S)$ holds trivially, but $\psi([a]^{0\msb})=-1$. This shows that the second requirement in condition \textit{(i)} of Theorem~\ref{theorem:filteredPositivstellsatz} cannot be omitted and the characters of the symbol algebra are really needed.
\end{example}

\begin{example}
  Let  $\A=\RR[x,y]$ and let $\S$ be the semiring generated by $x$, $y$, and define the $\S$-module $\C \coloneqq \S + (xy-1) \S$.
  Note that $\C$ is a (filtered) simplicial convex cone, and in fact is generated by the linearly independent set 
  $\set{x^k (xy-1)^\ell}{k,\ell\in \NN_0} \cup \set{y^k (xy-1)^\ell}{k,\ell\in \NN_0}$.
  In this example we consider two  different filtrations of $\A$. This leads to two different Positivstellens\"atze, with different conditions of strict
  positivity at infinity and different denominators. We describe  only the corresponding sets of characters and leave
  the application of Theorem~\ref{theorem:filteredPositivstellsatz} to the reader.

  {\it First filtration}\\
  We proceed as in  Example~\ref{example:filtrationsbydegree} and assign the degree $1$ to both variables $x$ and $y$. This
  yields the symbol algebra $\RR[x,y]^\sb \cong \RR[x,y]$ and the symbol map $[\argument]^{k\msb} \colon \RR[x,y]^{(k)} \to \RR[x,y]^\sb$
  assigns the symbol $[p]^{k\msb} \coloneqq \sum_{i,j \in \NN_0, i+j = k} p_{i,j} x^i y^j$ to an element $p = \sum_{i,j \in \NN_0, i+j\le k} p_{i,j} x^i y^j\in \RR[x,y]^{(k)}$.
  One can check that $s \coloneqq 1+x+y \in \Unit + \RR[x,y]^{(1)}$ is compatible with this filtration (Lemma~\ref{lemma:compatibility} applies). 
    
  Obviously, the characters of $\RR[x,y]^\sb \cong \RR[x,y]$ can be identified with the evaluation functionals $\chi_{\xi,\eta}$ at points $(\xi,\eta) \in \RR^2$,
  A character $\chi_{\xi,\eta}$ is $\S^\sb$-positive if and only if
  $0 \le \chi_{\xi,\eta}([x]^{1\msb}) = \xi$ and $0 \le \chi_{\xi,\eta}([y]^{1\msb}) = \eta$. The additional condition $\chi_{\xi,\eta}([s]^{1\msb}) = 1$
  is equivalent to $\xi + \eta = 1$. Thus,
  \begin{align*} 
    \set[\big]{ \psi \in \characters(\RR[x,y]^\sb,\S^\sb) }{ \psi([s]^{1\msb}) = 1 }
    =
    \set[\big]{ \chi_{\lambda,1-\lambda} }{ \lambda \in {[0,1]} }
    .
  \end{align*}
  
  {\it Second filtration}\\
  Another possibility is to define $\RR[x,y]^{(k)}$ as the  span of the monomials $x^i y^j$ with $i,j \in \{0,\dots,k\}$.
  Then $\RR[x,y]^{(k)}_0 = \RR[x,y]^{(k-1)}$ for all $k\in \NN$, and $\RR[x,y]^{(k)} / \RR[x,y]^{(k)}_0$ can be identified with the
  span of the monomials $x^k y^i$ and $x^i y^k$ for $i\in \{0,\dots,k\}$.
  In this case, $s \coloneqq (1+x)(1+y) \in \RR[x,y]^{(1)}$ is compatible with the filtration by Lemma~\ref{lemma:compatibility}.
  Now the product on $\RR[x,y]^\sb$ is more complicated: 
  $\RR[x,y]^\sb$ is generated by $[x]^{1\msb}$, $[y]^{1\msb}$ and $[xy]^{1\msb}$, and $[x]^{1\msb} [y]^{1\msb} = [xy]^{2\msb} = 0$ holds.
  Hence the characters of $\RR[x,y]^\sb$  are the evaluation functionals $\chi_{\xi,\eta,\rho}$ with $\xi,\eta,\rho \in \RR$, $\xi\eta = 0$,
  given by $\chi_{\xi,\eta,\rho}([x]^{1\msb}) = \xi$, $\chi_{\xi,\eta,\rho}([y]^{1\msb}) = \eta$, and $\chi_{\xi,\eta,\rho}([xy]^{1\msb}) = \rho$.
  If such a functional is positive, then $\xi,\eta,\rho \in {[0,\infty)}$, and $\chi_{\xi,\eta,\rho}([s]^{1\msb}) = 1$ requires $\xi + \eta + \rho = 1$. Therefore, 
  \begin{equation*} 
    \set[\big]{ \psi \in \characters(\RR[x,y]^\sb,\S^\sb) }{ \psi([s]^{1\msb}) = 1 }
    =
    \set[\big]{ \chi_{0,\lambda,1-\lambda},\, \chi_{\lambda,0,1-\lambda} }{ \lambda \in {[0,1]} }.
  \end{equation*}
\end{example}

The following is an application to an algebra that is not finitely generated.
\begin{example}
  Let $\A$ be the linear span of functions $\RR \ni \xi \mapsto E_\alpha(\xi) \coloneqq e^{\alpha \xi} \in \RR$, $\alpha \in {[0,\infty)}$.
  Clearly, $\A$ is a unital subalgebra of the algebra of continuous functions on $\RR$, 
  the functions $\set{E_\alpha}{\alpha \in {[0,\infty)}}$ form a basis of $\A$, and $\Unit = E_0$.
  Let $\S$ be the semiring  of $\A$  generated by $\set{E_\alpha-\Unit}{\alpha \in {[0,\infty)}}$. It is clear that $\S$ is generating.

  We fix a number $\beta \in {(0,\infty)}$ and define a filtration of $\A$ by letting $\A^{(k)}$, $k\in \NN$, be 
  the linear subspace spanned by $\set{E_\alpha}{\alpha \in {[0,k\beta]}}$.
  Then $\A^{(k)}_0$ is spanned by $\set{E_\alpha}{\alpha \in {[0,k\beta)}}$,
  and under the identification $\A^{(k)} / \A^{(k)}_0 \cong E_{k\beta} \RR$, the symbol algebra $\A^\sb$ is isomorphic to the unital subalgebra of $\A$  generated by $E_\beta$.
  Moreover, $(E_\beta)^k - E_\alpha = (E_\alpha - \Unit) (E_{k\beta-\alpha} - \Unit) +  (E_{k\beta-\alpha} - \Unit) \in \S$
  and $(E_\beta)^k + E_\alpha = (E_{k\beta} - \Unit) + (E_\alpha - \Unit) + 2 \Unit \in \S$ for $\alpha \in {[0,k\beta]}$
  show that $E_\beta$ is compatible with this filtration.
  Next we verify that $\S$-positive characters on $\A$ are evaluation functionals,
  \begin{equation}\label{EbetaC}
    \characters(\A,\S) = \set[\big]{ \chi_\xi }{ \xi\in {[0,\infty)} } .
  \end{equation}
  It is obvious that $\chi_\xi\in \characters(\A,\S)$ for $\xi \in [0,\infty)$. 
  Conversely, let $\chi \in \characters(\A, \S)$ be given.
  Then $\chi(E_\alpha) = \chi(E_\alpha - \Unit) + 1 \in {[1,\infty)}$ for all $\alpha \in {[0,\infty)}$, so that 
  ${[0,\infty)} \ni \alpha \mapsto \Xi(\alpha) \coloneqq \ln(\chi(E_\alpha)) \in {[0,\infty)}$ is a well-defined function.
  Moreover, $\Xi(0) = 0$ because $E_0 = \Unit$ and $\Xi(\alpha + \alpha') = \Xi(\alpha) + \Xi(\alpha')$ because $E_{\alpha+\alpha'} = E_\alpha E_{\alpha'}$.
  In particular, this  implies that $\Xi$ is monotonically increasing. Set $\xi \coloneqq \Xi(1)$. Then $\Xi(n) = n\xi$ for all $n\in \NN_0$
  because $\Xi$ is additive. Hence $\Xi(q) = q\xi$ for rational $q \in {[0,\infty)}$ by multiplication with the denominator,
  and finally $\Xi(\alpha) = \alpha \xi$ for all $\alpha\in {[0,\infty)}$ because $\Xi$ is monotonically increasing. Then
  $\chi(E_\alpha) = e^{\Xi(\alpha)} = e^{\alpha \xi}= \chi_\xi(E_\alpha)$ for $\alpha\in {[0,\infty)}$, that is, $\chi=\chi_\xi$ which completes the proof of \eqref{EbetaC}.
  
  Let $f\in \A$. {\it If $f(\xi) > 0$ for all $\xi \in {[0,\infty)}$, then there exists $\gamma \in {[0,\infty)}$ such that $E_{\gamma} f \in \S$}.
  Indeed, we can express $f$ as $f = \sum_{i=1}^n f_i E_{\alpha_i}$, where $n\in \NN$, $0 \le \alpha_1 < \dots < \alpha_n$, and $f_i \in \RR\setminus\{0\}$ for 
  $i\in \{1,\dots,n\}$. If $n=1$ and $\alpha_1 = 0$, then $f_1 = f(\xi)>0$ for  $\xi \in {[0,\infty)}$ and therefore $f\in \S$.
  Otherwise we set $\beta \coloneqq \alpha_n \in {(0,\infty)}$. Then $f \in \A^{(1)}$ for the $\beta$-filtration of $\A$,
  and $[f]^{1\msb} = f_n E_\beta$ with $f_n > 0$. By Theorem~\ref{theorem:filteredPositivstellsatz} 
  this implies that $E_\gamma f = (E_\beta)^\ell f\in \S$ for sufficiently large $\ell \in \NN_0$ and with $\gamma \coloneqq \ell \beta$.
\end{example}

As another application of the filtered Positivstellensatz we prove a denominator-free Positivstellensatz for cylinders over simplices.
In  Section~\ref{onedimextension}, this result will be escalated to cylinders with arbitrary cross-sections.

\begin{proposition} \label{proposition:deltastreifen}
  Let $n\in \NN$, $\gamma \in {(0,\infty)}$ be given and let $\S$ be the semiring of $\RR[x_1,\dots,x_n]$ 
  that is generated by $x_1,\dots,x_n$ and $\gamma-\sum_{j=1}^n x_j$. Define the corresponding simplex in $\RR^n$,
  \begin{equation}\label{eq:Deltangamma}
    \Delta^n_\gamma
    \coloneqq
    \set[\Big]{ {(\xi_1,\dots,\xi_n)\in[0,\infty)}^{n} }{ \sum\nolimits_{j=1}^n \xi_j  \le \gamma }.
  \end{equation}
  Let $J$ be a closed subset of $\RR$ and $\Pos(J)$ be defined by \eqref{eq:posj}.
  Consider $p = \sum_{i=0}^{2k} p_i z^i \in \RR[x_1,\dots,x_n,z]$ with
  $p_i \in \RR[x_1,\dots,x_n]$ for  $i\in \{0,\dots,2k\}$.
  If $p(\xi,\eta) > 0$ for all $(\xi,\eta) \in \Delta^n_\gamma \times J$ and $p_{2k}(\theta) > 0$ for all $\theta \in \Delta^n_\gamma$,
  then there exists a number $\epsilon \in {(0,\infty)}$ such that $p \in \epsilon (1+z^2)^k + \S \otimes \Pos(J)$.
\end{proposition}
\begin{proof}
  First we note  that the semiring\,  $\S$\, of $\RR[x_1,\dots,x_n]$ is  Archimedean by Lemma~\ref{archhilfs}.
  
  We use the method and the natural identifications of Example~\ref{example:filtrationsbydegree} and consider the filtration
  of the algebra $\RR[x_1,\dots,x_n,z]$ that is obtained by assigning the degree $0$ to all $x_j$, $j\in \{1,\dots,n\}$, and the degree $1/2$ to $z$.
  Then the resulting symbol algebra $\RR[x_1,\dots,x_n,z]^\sb$ 
  is the unital subalgebra
  \begin{equation*}
    \RR[x_1,\dots,x_n,z]^\sb
    =
    \set[\Big]{
      \sum\nolimits_{j=0}^{\ell} q_{2j}\, z^{2j}
    }{
      \ell \in \NN_0;\,q_{2j} \in \RR[x_1,\dots,x_n] \textup{ for all }j\in \{0,\dots,\ell\}
    }
  \end{equation*}
  of $\RR[x_1,\dots,x_n,z]$ and $[p]^{k\msb} = p_{2k} \, z^{2k}$. Note that $\RR[x_1,\dots,x_n,z]^{(0)}\cong\RR[x_1,\dots,x_n]$ and that
  $\RR[x_1,\dots,x_n,z]^\sb$ is generated as a unital algebra by $\{z^2\} \cup \RR[x_1,\dots,x_n]$.
  
  Now we set $s \coloneqq 1+z^2 \in (1 + \S \otimes \Pos(J)) \cap \RR[x_1,\dots,x_n,z]^{(1)}$. Then $s$ is compatible with the filtration of $\RR[x_1,\dots,x_n,z]$:
  Indeed, as $\RR[x_1,\dots,x_n,z]$ is generated as a unital algebra by $\RR[x_1,\dots,x_n,z]^{(1)}$, Lemma~\ref{lemma:compatibility} applies and it
  suffices to show that for every $q \in \RR[x_1,\dots,x_n,z]^{(1)}$ there exists $\lambda \in {(0,\infty)}$ such that $\lambda s - q \in \S \otimes \Pos(J)$.
  We write $q \in \RR[x_1,\dots,x_n,z]^{(1)}$ as a sum $q = q_0+q_1 z + q_2 z^2$ with coefficients $q_0,q_1,q_2 \in \RR[x_1,\dots,x_n]$.  As the semiring $\S$ is Archimedean, for each $j\in \{0,1,2\}$ 
  there exists a number $\lambda_j' \in {(0,\infty)}$ such that $\lambda_j' - q_j \in \S$ and $\lambda_j' + q_j \in \S$.
  Since $z^2 \in \Pos(J)$, it  follows that  $\lambda_{2}' z^2 - q_2 z^2 \in \S \otimes \Pos(J)$.
  Since $(1+z^2) / 2  -z \in \Pos(J)$ and $(1+z^2) / 2  + z \in \Pos(J)$,  using Lemma~\ref{archhilfs}
  we find that $\lambda_1 ' (1+z^2) / 2 - q_1 z \in \S \otimes \Pos(J)$. Consequently, we have $\lambda s - q \in \S \otimes \Pos(J)$
   if we set $\lambda \coloneqq \lambda_1' / 2 + \max \{\lambda_0' , \lambda_2'\} \in {(0,\infty)}$.
  
  Note that $p\in \RR[x_1,\dots,x_n,z]^{(k)}$. Next we show that $p$ satisfies condition \refitem{item:filteredPositivstellsatz:ana} of Theorem~\ref{theorem:filteredPositivstellsatz}.
  First, each $\varphi \in \characters\big(\RR[x_1,\dots,x_n,z], \S \otimes \Pos(J)\big)$ is an evaluation functional at a point
  $(\xi,\eta) \in \Delta^n_\gamma \times J$, so $\varphi(p) = p(\xi,\eta) > 0$.
  Secondly, $[s]^{1\msb} = z^2 \in \RR[x_1,\dots,x_n,z]^\sb$, and as the symbol algebra $\RR[x_1,\dots,x_n,z]^\sb$ is generated by $z^2$ and $\RR[x_1,\dots,x_n]$,
  every character $\psi$ of $\RR[x_1,\dots,x_n,z]^\sb$ that fulfills $\psi([s]^{1\msb}) = 1$ is completely determined by its restriction to $\RR[x_1,\dots,x_n]$. This restriction
  is an evaluation functional at some point $\theta \in \RR^n$.
  If $\psi$  is also $(\S \otimes \Pos(J))^\sb$\,-positive, then $q(\theta) = \psi([q]^{0\msb}) \ge 0$ for all $q \in \S$, so $\theta \in \Delta^n_\gamma$.
  In this case, $\psi([p]^{k\msb}) = p_{2k}(\theta) > 0$.
  
  Then, by the implication \refitem{item:filteredPositivstellsatz:ana}$\implies$\refitem{item:filteredPositivstellsatz:alg}
  of Theorem~ \ref{theorem:filteredPositivstellsatz}, there exist $\epsilon \in {(0,\infty)}$ and $\ell \in \NN_0$ such that 
  \begin{equation*}
    (1+z^2)^\ell p \in \epsilon (1+z^2)^{k+\ell} + \S \otimes \Pos(J),
  \end{equation*}
  that is,~$(1+z^2)^\ell\big( p - \epsilon (1+z^2)^k \big) \in \S \otimes \Pos(J)$.
  Because $\S \otimes \Pos(J)$ is $(1+z^2)$\,-localizable by Corollary~\ref{corollary:examplesimplex},
  we conclude that $p - \epsilon (1+z^2)^k \in \S \otimes \Pos(J)$.
\end{proof}
If there exists an $\alpha \in \RR$ such that $\alpha < t$ for all $t\in J$, then the above Proposition~\ref{proposition:deltastreifen}
can be modified: In this case it is not necessary to assume the highest non-vanishing component of $p \in \RR[x_1,\dots,x_n,z]$
to have even degree in $z$, i.e.~one considers $p = \sum_{i=0}^k p_i z^i$ with $p_0, \dots, p_k \in \RR[x_1,\dots,x_k]$.
Then $p(\xi,\eta) > 0$ for all $(\xi,\eta) \in \Delta^n_\gamma \times J$ and $p_{k}(\theta) > 0$ for all $\theta \in \Delta^n_\gamma$
implies $p \in \epsilon (\alpha+z)^k + \S \otimes \Pos(J)$ for sufficiently small $\epsilon \in {(0,\infty)}$.
The proof can be adapted by working with the filtration of $\RR[x_1,\dots,x_n,z]$ obtained by assigning the degree $0$ to $x_1, \dots, x_n$
and the degree $1$ to $z$, and setting $s\coloneqq \alpha + z$. We do not carry out the details.

\section{Cylindrical extension of the Archimedean Positivstellensatz}\label{onedimextension}

Let $\S$ be a semiring of $\A$ and $J$ a   subset of $\RR$.
In this section we consider
the algebra $\A[z] \cong \A \otimes \RR[z]$  of polynomials over $\A$ in one variable $z$ and 
the semiring $\S \otimes \Pos(J)$ of $\A[z]$.
That is, $\S \otimes \Pos(J)$ is the convex cone of finite sums of elements $s p$, with $s\in \S$ and $p\in \RR[z]$ satisfying
$p(t) \ge 0$ for all $t\in J$. For instance, if $J=\RR$, then $\Pos(J) = \set{p_1^2 + p_2^2}{p_1,p_2 \in \RR[z]}$.
It is easily verified that the $\S \otimes \Pos(J)$\,-positive characters of $\A[z]$ are in one-to-one correspondence with 
pairs $(\varphi,\eta) \in \characters(\A,\S) \times J$ by assigning the character $\A[z] \ni a \mapsto \varphi(a(\eta)) \in \RR$
to any $(\varphi,\eta) \in \characters(\A,\S) \times J$.

The aim of this section is to prove the following Theorem~\ref{theorem:streifen}. It generalizes a result for (finitely generated)
preorderings by V.~Powers \cite{powers}. As in \cite{powers}, in the proof we use M. Schweighofer's trick \cite{schw02}  of reducing the problem to the simplex,
i.e., to Proposition~\ref{proposition:deltastreifen}. This technique works for arbitrary (not necessarily finitely generated)
semirings $\S$ of finitely generated algebras $\A$. It would be interesting to know whether Theorem~\ref{theorem:streifen} also holds for
not finitely generated algebras.

\begin{theorem} \label{theorem:streifen}  Let $\A$ be a finitely generated commutative unital\, $\RR$-algebra and $\S$  an Archimedean semiring of $\A$.
  Suppose that $J$ is an unbounded closed subset of $\RR$.
  Let $a = \sum_{i=0}^{2k} a_i z^i \in \A[z]$ with coefficients $a_0, \dots, a_{2k} \in \A$, $k\in \NN_0$. Then the following are equivalent:
  \begin{enumerate}
    \item\label{item:streifen:ana}
      $\varphi(a(\eta)) > 0$ for all $(\varphi,\eta) \in \characters(\A, \S) \times J$ and $\psi(a_{2k}) > 0$ for all $\psi \in \characters(\A,\S)$.
    \item\label{item:streifen:alg}
      There exists $\epsilon \in (0,\infty)$ such that $a \in \epsilon (\Unit+z^2)^k + \S \otimes \Pos(J)$.
  \end{enumerate}
\end{theorem}

For semibounded $J$,
a version of Theorem~\ref{theorem:streifen} demanding strict positivity in the highest (not necessarily even) degree can be obtained
by the same methods, see also the discussion after the proof of Proposition~\ref{proposition:deltastreifen}.

There are plenty of examples for Theorem \ref{theorem:streifen}; one may take all the examples of Archimedean semirings from Section \ref{examplesarch} as $\S$.

\begin{remark}
  Let\, $a = \sum_{i=0}^{2k} a_i z^i \in \A[z]$. If\, $\varphi(a(\eta)) \geq 0$ for all $(\varphi,\eta) \in \characters(\A, \S) \times J$
  and if $J$ is unbounded, then $\psi(a_{2k}) \geq 0$ for all $\psi \in \characters(\A,\S)$. Indeed,  $\psi(a(\eta))=\sum_{i=0}^{2k} \psi(a_i) \eta^i\geq 0$ 
  for $\eta\in \RR$ implies that 
  \begin{equation}\label{psi24}
    \lim_{|\eta|\to \infty, \, \eta \in J}~ \eta^{-2k}\sum\nolimits_{i=0}^{2k} \psi(a_i)\eta^i=\psi(a_{2k})\geq 0.
  \end{equation}
\end{remark}

Before we turn to the proof, we state an application of Theorem \ref{theorem:streifen} to the moment problem.
\begin{corollary}\label{onedimmp}
  Retain the assumptions and the notation of Theorem  \ref{theorem:streifen}. Let $L$ be a linear functional on $\A[z]$.
  If $L(a)\geq 0$ for all elements $a\in \S \otimes \Pos(J)$, then $L$ is a $(\characters(\A, \S) \times J)$-moment functional.
\end{corollary}
\begin{proof} 
  We apply Haviland's theorem (Proposition \ref{haviland}) to the algebra $\A[z]$ and  $\K\coloneqq\characters(\A, \S) \times J$
  and verify condition \refitem{haviland:ddagger} therein. 
  For let  $a = \sum_{i=0}^{2k} a_i z^i \in \Pos(\K)$. Set $b\coloneqq(1+z^2)^k\in \Pos(\K)$
  and $a_{\varepsilon'}\coloneqq a+\varepsilon' b$ for $\varepsilon'>0$. 
  Then $\varphi(a_{\varepsilon'}(\eta))=\varphi(a(\eta))+\varepsilon'(1+\eta^2)^k\geq 0$ for $(\varphi,\eta) \in \characters(\A, \S) \times J$ 
  and $\psi((a_{\varepsilon'})_{2k})= \psi(a_{2k} +\varepsilon'\Unit)=\psi(a_{2k}) +\varepsilon'\geq \varepsilon'>0$
  by \eqref{psi24}. This shows that condition \refitem{item:streifen:ana} of 
  Theorem  \ref{theorem:streifen} is fulfilled. Therefore, $a_{\varepsilon'}\in \S \otimes \Pos(J)$ by Theorem \ref{theorem:streifen}, so that
  $L(a_{\varepsilon'})\geq 0$ by assumption.
\end{proof}

In order to escalate Proposition~\ref{proposition:deltastreifen} to Theorem~\ref{theorem:streifen}, we  need some preliminary lemmas and some more notation:
For any subset $S\subseteq \RR[x_1,\dots,x_n]$ and $\gamma \in {(0,\infty)}$ we write 
\begin{equation}
  \Delta^n_\gamma(S) \coloneqq \set[\Big]{\xi=(\xi_1,\dots,\xi_n) \in {[0,\infty)}^{n}}{\sum\nolimits_{j=1}^n \xi_j\le \gamma \textup{ and }s(\xi) \ge 0\textup{ for all }s\in S}
  .
\end{equation}
Note that $\Delta^n_\gamma(\emptyset) = \Delta^n_\gamma$ as in \eqref{eq:Deltangamma} and that
$\Delta^n_\gamma(S)$ is compact for all $S\subseteq \RR[x_1,\dots,x_n]$ because $\Delta^n_\gamma(S) \subseteq \Delta^n_\gamma$.

The semiring $\S$ in Theorem \ref{theorem:streifen} is not necessarily finitely generated. The next lemma will be used
to reduce the considerations to finitely generated semirings.

\begin{lemma} \label{lemma:ArchimedeanSemiringsOnFinitelyGeneratedAlgebras}
  Let $\gamma\in {(0,\infty)}$ and let $S$ be a subset of $\RR[x_1,\dots,x_n]$, $n\in \NN$.
  Then, for each polynomial $p \in \RR[x_1,\dots,x_n]$ that fulfills $p(\xi) > 0$ for all $\xi \in \Delta^n_\gamma(S)$,
  there exists a finite subset $F$ of $S$ such that $p(\xi) > 0$ for all $\xi \in \Delta^n_\gamma(F)$.
\end{lemma}
\begin{proof}
  As $\Delta^n_\gamma(S)$ is compact, $\epsilon \coloneqq \frac{1}{2}\min \set{ p(\xi) }{\xi \in \Delta^n_\gamma(S)}$ is well-defined and $\epsilon >0$
  because $p(\xi) > 0$ for all $\xi \in \Delta^n_\gamma(S)$ by assumption. Note that $\Delta^n_\gamma(\{ \epsilon \Unit - p\}) \cap \Delta^n_\gamma(S) = \emptyset$ because 
  $p(\xi) \le \epsilon$ for $\xi \in \Delta^n_\gamma(\{ \epsilon \Unit - p\})$ while $p(\xi) \ge 2 \epsilon$ for $\xi \in \Delta^n_\gamma(S)$.
  This means that for every $\xi \in \Delta^n_\gamma(\{ \epsilon \Unit - p\})$ there exists $s_\xi \in S$ such that $s_\xi(\xi) < 0$.
  Set $U_\xi \coloneqq \set{\theta \in \RR^n}{ s_\xi(\theta) < 0}$. Then each set $U_\xi$, $\xi \in \Delta^n_\gamma(\{ \epsilon \Unit - p\})$,
  is an open neighbourhood of $\xi$. Therefore, $\bigcup_{\xi \in \Delta^n_\gamma(\{ \epsilon \Unit - p\})} U_\xi \supseteq \Delta^n_\gamma(\{ \epsilon \Unit - p\})$.
  By the compactness of $\Delta^n_\gamma(\{ \epsilon \Unit - p\})$ there exist $k \in \NN_0$ and $\xi'_1, \dots \xi'_k \in \Delta^n_\gamma(\{ \epsilon \Unit - p\})$
  such that $U_{\xi'_1} \cup \dots \cup U_{\xi'_k} \supseteq \Delta^n_\gamma(\{ \epsilon \Unit - p\})$.
  Set $F \coloneqq \{ s_{\xi'_1}, \dots, s_{\xi'_k} \}$. Then $F$ is a finite subset of $S$ such that $\Delta^n_\gamma(\{ \epsilon \Unit - p\}) \cap \Delta^n_\gamma(F) = \emptyset$
  by construction. Hence $p(\xi) > \epsilon > 0$ for all $\xi \in \Delta^n_\gamma(F)$.
\end{proof}

We also need the following observation for functions on compact topological spaces:

\begin{lemma} \label{lemma:addingstuff}
  Let $C$ be a compact topological space and $f,g \colon C \to \RR$ continuous.
  If $g(\xi) \ge 0$ for all $\xi \in C$ and  $f(\xi)> 0$ for all those $\xi \in C$ for which $g(\xi) = 0$, then there exists
  $\lambda \in {(0,\infty)}$ such that $f(\xi) + \lambda g(\xi) > 0$ for all $\xi \in C$.
\end{lemma}
\begin{proof}
  The subset $Z \coloneqq \set{\xi \in C}{g(\xi) = 0}$ of $C$ is closed, hence compact. Therefore the minimum
  $\mu_{f,Z} \coloneqq \min\set{f(\xi)}{\xi \in Z}$ exists and $\mu_{f,Z} > 0$ because $f(\xi)> 0$ for all $\xi \in Z$ by assumption.
  Similarly, the subset $W \coloneqq \set{\xi \in C}{ f(\xi) \le \mu_{f,Z} / 2}$ of $C$ is also  closed and therefore compact. Hence the minimum
  $\mu_{g,W} \coloneqq \min\set{g(\xi)}{\xi \in W}$ exists. Moreover, $W \cap Z = \emptyset$ because $f(\xi) \le \mu_{f,Z} / 2$ for all $\xi \in W$
  while $f(\xi) \ge \mu_{f,Z}$ for all $\xi \in Z$. For any $\xi \in W$ this implies $g(\xi) \neq 0$ by definition of $Z$, hence $g(\xi) > 0$ for all $\xi \in W$,
  and especially $\mu_{g,W} > 0$.
  Finally, the minimum $\mu_{f,C} \coloneqq \min\set{f(\xi)}{\xi \in C}$ exists because $C$ is compact. Thus, if we set $\lambda \coloneqq 1 + \abs{\mu_{f,C}} / \mu_{g,W} \in {(0,\infty)}$,
  then $f(\xi) + \lambda g(\xi) > 0$ holds for all $\xi \in C$:
  If $\xi \in W$, then $f(\xi) + \lambda g(\xi) \ge \mu_{f,C} + \lambda \mu_{g,W} \ge \mu_{g,W} > 0$, and if $\xi \in C \setminus W$,
  then $f(\xi) + \lambda g(\xi) \ge f(\xi) > \mu_{f,Z} / 2 > 0$ by the definition of $W$.
\end{proof}

\begin{lemma} \label{lemma:properpolynomial}
  Let $p = \sum_{j=0}^{2k} p_j z^j \in \RR[x_1,\dots,x_n,z]$, $n\in \NN$,  with $p_0, \dots, p_{2k} \in \RR[x_1,\dots,x_n]$, $k\in \NN_0.$
  Moreover, let $S$ be a subset of $\RR[x_1,\dots,x_n]$ and $\gamma \in {(0,\infty)}$. If $p_{2k}(\xi) > 0$ for all $\xi \in \Delta^n_\gamma(S)$,
  then $\set{(\xi,\eta)\in\Delta^n_\gamma(S)\times J}{p(\xi,\eta) \le \lambda}$ is compact for all $\lambda \in \RR$.
\end{lemma}
\begin{proof}
  As $\Delta^n_\gamma(S)$ is compact, the minimum $\epsilon \coloneqq \min \set{p_{2k}(\xi)}{\xi \in \Delta^n_\gamma(S)}$ exists,
  and if $p_{2k}(\xi) > 0$ for all $\xi \in \Delta^n_\gamma(S)$, then $\epsilon \in {(0,\infty)}$.
  Similarly, there  exists $\nu \in {(0,\infty)}$ such that $\abs{p_j(\xi)} \le \nu$ for all $\xi \in \Delta^n_\gamma(S)$ and all $j\in\{0,\dots,2k-1\}$.
  Consequently,  for all $(\xi,\eta) \in \Delta^n_\gamma(S) \times J$ we have
  \begin{equation*}
    \abs[\big]{p(\xi,\eta)}
    \ge
    p_{2k}(\xi) \,\eta^{2k} - \sum_{j=0}^{2k-1} \abs{p_j(\xi)} \,\abs{\eta}^j
    \ge
    \epsilon \eta^{2k} -  \nu \sum_{j=0}^{2k-1} \abs{\eta}^j.
  \end{equation*}
 Now let $\lambda \in \RR$.
  Then there is a number $\rho \in {(0,\infty)}$ such that $\epsilon \eta^{2k} -  \nu \sum_{j=0}^{2k-1} \abs{\eta}^j > \lambda$
  for all $\eta \in J$ with $\abs{\eta} > \rho$. Hence
  $\set{(\xi,\eta)\in\Delta^n_\gamma(S)\times J}{p(\xi,\eta) \le \lambda}$ is a closed subset of $\Delta^n_\gamma(S) \times {[-\rho,\rho]}$
  and therefore it is compact.
\end{proof}

\begin{proof}[of Theorem~\ref{theorem:streifen}]
  \refitem{item:streifen:alg}$\implies$\refitem{item:streifen:ana}:  Assume that $a \in \epsilon (\Unit+z^2)^k + \S \otimes \Pos(J)$ for some $\epsilon >0$. Then 
   $\varphi(a(\eta)) \ge \epsilon (1+\eta^2)^k > 0$  for all $(\varphi,\eta) \in \characters(\A, \S) \times J$. Hence, for $\psi \in \characters(\A,\S)$ we obtain
  \begin{equation*}
    \psi(a_{2k}) = \lim_{|\eta| \to \infty,\, \eta \in J} ~ \eta^{-2k}\psi(a(\eta)) \ge  \epsilon \lim_{|\eta| \to \infty, \, \eta \in J}~ \eta^{-2k}(1+\eta^2)^k = \epsilon>0.
  \end{equation*}

  \refitem{item:streifen:ana}$\implies$\refitem{item:streifen:alg}:
  Assume now that $\varphi(a(\eta)) > 0$ for all $(\varphi,\eta) \in \characters(\A, \S) \times J$ and  $\psi(a_{2k}) > 0$ for all $\psi \in \characters(\A,\S)$.
  
  The first step is to lift $a$ and $\S$ to a polynomial algebra:
  Let $g_1, \dots, g_n \in \A$, $n\in \NN$, be a finite set of generators of $\A$. As $\S$ is Archimedean, there exist
  $\lambda_1, \dots, \lambda_n \in {(0,\infty)}$ such that $\lambda_j \Unit - g_j \in \S$ for  $j\in \{1,\dots,n\}$.
  We  define a unital algebra homomorphism  $\Phi \colon \RR[x_1,\dots,x_n] \to \A$ by $\Phi(x_j) \coloneqq \lambda_j \Unit - g_j$ for  $j\in \{1,\dots,n\}$.
  Clearly, $\Phi$ is surjective, so we find $\hat a_0, \dots, \hat a_{2k} \in \RR[x_1,\dots,x_n]$ such that $a_i = \Phi(\hat a_i)$ for all $i \in \{0,\dots,2k\}$.
 Now we define $\hat a \coloneqq \sum_{i=0}^{2k} \hat a_i z^i \in \RR[x_1,\dots,x_n,z]$.
  Then the preimage $\T \coloneqq \Phi^{-1}(\S)$  of $\S$ under $\Phi$ is an Archimedean semiring of $\RR[x_1,\dots,x_n]$ because $\S$ is Archimedean by assumption.
  Moreover, $x_j \in \T$ for  $j\in \{1,\dots,n\}$ by construction. Also, there  exists a number $\gamma \in {(0,\infty)}$ such that $\gamma\Unit - \sum_{j=1}^n x_j \in \T$.
  
  The next step is to reduce $\T$ to a finite set:
  Note that $\ker \Phi \subseteq \T$ by definition, so $\ker \Phi \subseteq \T \cap (-\T)$, and
  $q(\xi) = 0$ for all $q\in \ker \Phi$ and $\xi \in \Delta^n_\gamma(\T)$. Therefore, for every $\xi \in \Delta^n_\gamma(\T)$ there 
  exists a unique well-defined $\Phi_{*;\xi} \in \characters(\A,\S)$ given by $\Phi_{*;\xi}(\Phi(q)) = q(\xi)$ for all $q\in \RR[x_1,\dots,x_n]$.
  Consequently,
  \begin{equation*}
    \hat a(\xi,\eta) = \sum_{i=0}^{2k} \hat a_i(\xi) \eta^i = \sum_{i=0}^{2k} \Phi_{*;\xi}(a_i) \eta^i = \Phi_{*;\xi}(a(\eta)) > 0
  \end{equation*}
  for all $(\xi, \eta) \in \Delta^n_\gamma(\T) \times J$ and $\hat a_{2k}(\xi) = \Phi_{*;\xi}(a_{2k}) > 0$
  for all $\xi \in \Delta^n_\gamma(\T)$.
  Applying Lemma~\ref{lemma:ArchimedeanSemiringsOnFinitelyGeneratedAlgebras} to $\hat a_{2k}$ yields a finite subset $F'_\infty$ of $\T$
  such that $\hat a_{2k}(\xi) > 0$ holds for all $\xi \in \Delta^n_\gamma(F'_\infty)$.
  By Lemma~\ref{lemma:properpolynomial}, the set $\set{(\xi,\eta)\in \Delta^n_\gamma(F'_\infty) \times J}{\hat a(\xi,\eta) \le 1}$ is compact,
  so there exists $\rho \in {(0,\infty)}$ such that $\hat a(\xi,\eta)> 1$ for all $(\xi,\eta)\in \Delta^n_\gamma(F'_\infty) \times J$ with $\abs{\eta} > \rho$.
  Similarly, applying Lemma~\ref{lemma:ArchimedeanSemiringsOnFinitelyGeneratedAlgebras} to $\hat a(\argument,\eta) \in \RR[x_1,\dots,x_n]$ with arbitrary $\eta \in [-\rho,\rho] \cap J$
  yields finite subsets $F'_\eta$ of $\T$ such that $\hat a(\xi,\eta) > 0$ for all $\xi \in \Delta^n_\gamma(F'_\eta)$.
  As $\Delta^n_\gamma(F'_\eta)$ is compact and  the map $\Delta^n_\gamma(F'_\eta) \times \RR \ni (\xi,\mu) \to \hat a(\xi,\mu) \in \RR$ is continuous,
  there  exists an open neighbourhood $U_\eta$ of $\eta$ such that $\hat a(\xi,\mu) > 0$ for all $(\xi,\mu) \in \Delta^n_\gamma(F'_\eta)\times U_\eta$.
  By compactness of $[-\rho,\rho]\cap J$ there exist finitely many $\hat\eta_1, \dots, \hat\eta_\ell \in [-\rho,\rho] \cap J$, $\ell \in \NN_0$, such that
  $\bigcup_{j=1}^\ell U_{\hat\eta_j} \supseteq [-\rho,\rho]\cap J$. We set $F \coloneqq F'_\infty \cup \bigcup_{j=1}^\ell F'_j$. Then $\hat a_{2k}(\xi) > 0$ for all
  $\xi \in \Delta^n_\gamma(F)$ and $\hat a(\xi,\eta) > 0$ for all $(\xi,\eta) \in \Delta^n_\gamma(F) \times J$. Note that $F$ is a finite set by construction,
  say $F = \{f_1,\dots,f_m\}$ with $m \in \NN_0$.
  
  A modification of a trick used in
  \cite{averkov} will allow us to apply Proposition~\ref{proposition:deltastreifen}
  by introducing additional variables: As $\T$ is Archimedean, there exists $\delta \in {(0,\infty)}$ such that $\delta\Unit - \sum_{j=1}^m f_j \in \T$.
  We set
  \begin{align*}
    \Delta^{n+m}_{\gamma+\delta}
    &\coloneqq
    \set[\Big]{
      (\xi,\theta) \in {(0,\infty)}\,^n \times {(0,\infty)}\,^m
    }{
      \sum\nolimits_{j=1}^n \xi_j + \sum\nolimits_{j=1}^m \theta_j \le \gamma+\delta
    },
    \\
    Z
    &\coloneqq
    \set[\big]{
      (\xi,\theta) \in \Delta^{n+m}_{\gamma+\delta}
    }{
      f_j(\xi) = \theta_j \textup{ for all }j \in \{1,\dots,m\}
    }
    .
  \end{align*}
  Note that all $(\xi,\theta) \in Z$ fulfill $\xi \in \Delta^n_\gamma(F)$.
  We  define $p \coloneqq \sum_{j=1}^m (f_j - y_j)^2 \in \RR[x_1,\dots,x_n,y_1,\dots,y_m]$.
  Clearly, $p(\xi,\theta) \ge 0$ for all $(\xi,\theta) \in \Delta^{n+m}_{\gamma+\delta}$, and if $(\xi,\theta) \in \Delta^{n+m}_{\gamma+\delta}$ fulfills $p(\xi,\theta) = 0$,
  then $(\xi,\theta) \in Z$ and in particular $\xi \in \Delta^n_\gamma(F)$. This implies that $\hat a(\xi,\eta) > 0$ for all $\eta \in \RR$ and $\hat a_{2k}(\xi) > 0$.
  Lemma~\ref{lemma:addingstuff}, applied to the continuous functions $\Delta^{n+m}_{\gamma+\delta} \ni (\xi,\theta) \mapsto \hat a_{2k}(\xi) \in \RR$ and
  $\Delta^{n+m}_{\gamma+\delta} \ni (\xi,\theta) \mapsto p(\xi,\theta) \in \RR$, now shows that there exists $\lambda_\infty \in {(0,\infty)}$ such that
  \begin{equation} \label{eq:theorem:streifen:intern1}
    \hat a_{2k}(\xi) + \lambda_\infty p(\xi,\theta) > 0
    \tag{\textup{$*$}}
  \end{equation}
  holds for all $(\xi,\theta) \in \Delta^{n+m}_{\gamma+\delta}$.
 Therefore, Lemma \ref{lemma:properpolynomial}  applies to $\hat a + \lambda_\infty p z^{2k} \in \RR[x_1,\dots,x_n,y_1,\dots,y_m,z]$ and shows that
  $C \coloneqq \set{(\xi,\theta,\eta) \in \Delta^{n+m}_{\gamma+\delta} \times J}{\hat a(\xi,\eta) + \lambda_\infty p(\xi,\theta) \eta^{2k} \le 1}$
  is compact. By Lemma~\ref{lemma:addingstuff}, now applied to the continuous functions $C \ni (\xi,\theta,\eta) \mapsto \hat a(\xi,\eta) +\lambda_\infty p(\xi,\theta) \eta^{2k} \in \RR$ and
  $C \ni (\xi,\theta,\eta) \mapsto p(\xi,\theta) \in \RR$,  there exists a number $\lambda_0 \in {(0,\infty)}$ such that
  \begin{equation} \label{eq:theorem:streifen:intern2}
    \hat a(\xi,\eta) +\lambda_\infty p(\xi,\theta) \eta^{2k} + \lambda_0 p(\xi,\theta) > 0
    \tag{\textup{$**$}}
  \end{equation}
   for all $(\xi,\theta,\eta) \in C$, hence even for all $(\xi,\theta,\eta) \in \Delta^{n+m}_{\gamma+\delta} \times J$ by the definition of $C$.
  
  Finally, applying Proposition~\ref{proposition:deltastreifen} to $\hat a + \lambda_\infty p z^{2k} + \lambda_0 p \in \RR[x_1,\dots,x_n,y_1,\dots,y_m,z]$
  with estimates \eqref{eq:theorem:streifen:intern1} and \eqref{eq:theorem:streifen:intern2}, we obtain a number
  $\epsilon>0$ such that $\hat a + \lambda_\infty p z^{2k} + \lambda_0 p \in \epsilon (1+z^2)^k + \S_{\Delta;\gamma+\delta} \otimes \Pos(J)$,
  where $\S_{\Delta;\gamma+\delta}$ is the semiring of $\RR[x_1,\dots,x_n,y_1,\dots,y_m]$ that is generated by $x_1,\dots,x_n, y_1,\dots,y_m$ and 
  $\gamma+\delta-\sum_{j=1}^n x_j - \sum_{j=1}^m y_j$.
  Now let $\Psi \colon \RR[x_1,\dots,x_n,y_1,\dots,y_m,z] \to \A[z]$ be the unital algebra homomorpism  defined by 
  $\Psi(x_j) \coloneqq \Phi(x_j)$ for  $j\in \{1,\dots,n\}$, $\Psi(y_j) \coloneqq \Phi(f_j)$ for  $j\in \{1,\dots,m\}$, and $\Psi(z) \coloneqq z$.
  Then we have $\Psi(\hat a) = \Phi(\hat a) = a$ and $\Psi(f_j-y_j) = \Phi(f_j) - \Phi(f_j) = 0$ for  $j\in \{1,\dots,m\}$, hence $\Psi(p) = 0$,
  which shows that $\Psi(\hat a + \lambda_\infty p z^{2k} + \lambda_0 p) = a$.
  Moreover, $\Psi(x_j) = \Phi(x_j) = \lambda_j \Unit - g_j \in \S$  for  $j\in \{1,\dots,n\}$. Similarly $\Psi(y_j) = \Phi(f_j) \in \S$ because $f_j \in \T = \Phi^{-1}(\S)$
  for  $j\in \{1,\dots,m\}$, and
  \begin{equation*}
    \Psi\bigg(\gamma+\delta-\sum_{j=1}^n x_j - \sum_{j=1}^m y_j\bigg) = \Phi\Big( \underbrace{\gamma + \delta - \sum\nolimits_{j=1}^n x_j - \sum\nolimits_{j=1}^m \Phi(f_j)}_{\in \T} \Big) \in \S
  \end{equation*}
  by the construction of $\gamma, \delta,$ and $\T$. This shows that the image of $\S_{\Delta;\gamma+\delta}$ under $\Psi$ is a subset of $\S$, so
  $a = \Psi(\hat a + \lambda_\infty p z^{2k} + \lambda_0 p) \in \epsilon (1+z^2)^k + \S\otimes \Pos(J)$.
\end{proof}

\section*{Acknowledgment}
The research of M. S.\, was supported by the Deutsche Forschungsgemeinschaft (SCHM1009/6-1).

\end{onehalfspace}

\bibliographystyle{amsalpha}

\begin{thebibliography}{A}
\bibitem[Av13]{averkov}
Averkov, G.: Constructive proofs of some Positivstellens{\"a}tze for
  compact semialgebraic subsets of $R^d$, J. Optim. Theory
  Appl. \textbf{158} (2013), 410--418.

\bibitem[BS83]{beckers}
Becker, E. and N. Schwartz: Zum Darstellungssatz von Kadison--Dubois, Arch. Math.
 {\bf 40}(1983), 42--428.
\bibitem[C09]{cimpric09} 
Cimpri\v{c}, J.: A representation theorem for quadratic modules on $*$-rings, Canad. Math. Bulletin {\bf 52}(2009), 39--52.
\bibitem[CMN09]{cmn} 
Cimpri\v{c}, J., M. Marshall and T. Netzer: Closures of quadratic modules, Israel J. Math. {\bf 183}(2011), 445–-474.
\bibitem[H36]{Haviland}
Haviland, E. K.:
On the momentum problem for distribution functions in more than one
dimension: I. Amer.\ J. Math.\ {\bf 57}(1935), 562--582, and~ II. {\bf 58}(1936), 164--168.
\bibitem[J01]{jacobi} 
Jacobi, T.: A representation theorem for certain partially ordered commutative rings, Math. Z. {\bf 23}(2001), 259--273.
\bibitem[J42]{jung} 
Jung, H.W.E., \"Uber ganze birationale Transformationen der Ebene, J. reine Angew. Math. {\bf 184}(1942), 161--174.
\bibitem[Kv64]{krivine}
Krivine, J.L.: Quelques propri\'et\'es des pr\'eordres dans les anneaux commutatifs unitaires,
C. R. Acad. Sci. Paris {\bf 258}(1964), 3417--3418.

\bibitem[KM02]{kumarshall}
Kuhlmann, S. and M. Marshall: Positivity, sums of squares and the multi-dimensional moment problem, Trans. Amer. Math. Soc. {\bf 354}(2002),  4285--4302.

\bibitem[KMS05]{kumsch}
Kuhlmann, S.,  M. Marshall, and N. Schwartz: Positivity, sums of squares and the multi-dimensional moment problem. II, Adv. Geom. {\bf 5}(2005), 583--606.

\bibitem[M01]{M2}  Marshall, M.: Extending the Archimedean Positivstellensatz to the non-compact case. Bull. Canad. Math. Soc. {\bf{44}} (2001), 223--230.

\bibitem[M08]{marshall} Marshall, M.: {\it Positive Polynomials and Sums of Squares}, Math. Surveys and Monographs, Amer. Math. Soc., Providence, R.I., 2008.


\bibitem[M10]{marshall10} Marshall, M.: Polynomials non-negative on a strip, Proc. Amer. Math. soc. {\bf 138}(2010), 1559--1567.

\bibitem[N09]{netzer} Netzer, T.: Positive polynomials and sequential closures of quadratic modules, Trans. Amer. Math. Soc. {\bf 362}(2009), 2619--2639.

\bibitem[Po04]{powers} Powers, V.: Positive polynomials and the moment problem for cylinders with compact cross-section, J. Pure Applied Algebra {\bf 188}(2004), 217--226.
\bibitem[Pu94]{put94}
Putinar, M.: Positive polynomials on compact semialgebraic sets, Indiana Univ. Math. J. {\bf 42}(1994),  969--984.

\bibitem[PD]{PD} A. Prestel and C.N. Delzell: \textit{Positive Polynomials}, Monographs in Math., Springer-Verlag, Berlin, 2001.

\bibitem[Sch09]{scheiderer}
 Scheiderer, C.: Positivity and sums of squares: A guide to recent results, In: {\it Emerging Applications of Algebraic Geometry},  M. Putinar and S. Sullivant (Editors), Springer-Verlag, New York,  2009, 271--324. 

\bibitem[S91]{sch91}  Schm\"udgen, K.: The $K$-moment problem for compact semialgebraic sets, Math. Ann. {\bf{289}}(1991), 203--206.
\bibitem[S09]{sch09}
Schm\"udgen, K.: Noncommutative real algebraic geometry--some basic concepts and first ideas, In: {\it Emerging Applications of Algebraic Geometry},  M. Putinar and S. Sullivant (Editors), Springer, New York, 2009, 325--350.

\bibitem[S17]{sch17} Schm\"udgen, K.: {\it The Moment Problem}, Graduate Texts in Math., Springer-Verlag, Cham, 2017.

\bibitem[S20]{sch20} Schm\"udgen, K.: {\it An Invitation to Unbounded Representations of $*$-Algebras on Hilbert Space}, Graduate Texts in Math.,   Springer-Verlag, Cham, 2020.


\bibitem[Schw02]{schw02}  Schweighofer, M.: An algorithmic approach to   Schm\"udgen's Positivstellensatz, J. Pure Applied Algebra {\bf 166}(2002), 307--319.
\bibitem[Schw03]{schw03}  Schweighofer, M.: Iterated rings of bounded elements and generalizations of Schm\"udgen's Positivstellensatz, J. reine angew. Math. {\bf 554}(2003), 19--45.

\bibitem[W\"o98]{wm}
 W\"ormann, T.:  Strict positive Polynome in der semialgebraischen Geometrie, Dissertation, Universit\"at Dortmund, 1998.

\end{thebibliography}

\end{document}